\documentclass[11pt]{amsart}
\usepackage{amsmath,amssymb,latexsym,bm,xcolor}
\usepackage{amssymb,latexsym,amsfonts,verbatim,amscd,ifthen} 
\usepackage{amsthm}
\usepackage{amsmath}
\usepackage{ytableau} 
\usepackage{mathtools}
\usepackage[vcentermath]{youngtab}
\usepackage{mathrsfs}
\usepackage{color}
\usepackage{forest}
\usepackage{lscape}
\usepackage{tikz}
\usepackage[left=1in, right=1in, top=1in, bottom=1in]{geometry}
\usepackage{setspace}
\usepackage{pdfpages}

\usetikzlibrary{calc}

\newcount\tableauRow\newcount\tableauCol
\newcommand\ColumnTabloid[1]{%
  \begin{tikzpicture}[scale=0.5,draw/.append style={thick,black},baseline=4mm]
    \tableauRow=0
    \foreach \Row in {#1} {
       \tableauCol=1
       \foreach\k in \Row {
          \draw($(\the\tableauCol,\the\tableauRow)+(-.5,-.5)$)--++(0,1);
          \draw($(\the\tableauCol,\the\tableauRow)+(.5,-.5)$)--++(0,1);
          \draw(\the\tableauCol,\the\tableauRow)node{\k};
          \global\advance\tableauCol by 1
       }
       \global\advance\tableauRow by -1
    }
  \end{tikzpicture}
}
\tikzstyle{arrow} = [thick,->,>=stealth]

\usetikzlibrary{positioning}
\newtheorem{prop}{Proposition}[section]
\newtheorem{definition}{Definition}[section]
\newtheorem{theorem}{Theorem}[section]

\newtheorem{example}{Example}[section]
\newtheorem{lemma}{Lemma}[section]

\newenvironment{proof1}{{\emph{Proof of Theorem.} }}{\hfill $\Box$ \\}

\title[Combinatorial formulas for type $A$ Whittaker functions]{Relating three combinatorial formulas \\for type $A$ Whittaker functions}

\author[C.~Lenart]{Cristian Lenart}
\address[Cristian Lenart]{Department of Mathematics and Statistics, State University of New York at Albany, 
Albany, NY 12222, U.S.A.}
\email{clenart@albany.edu}

\author[J.~Sidoli]{James Sidoli}
\address[James Sidoli]{Department of Mathematics and Statistics, State University of New York at Albany, 
Albany, NY 12222, U.S.A.}
\email{jsidoli@albany.edu}
\thanks{Cristian Lenart was partially supported by the National Science Foundation grant DMS-1855592 and the Simons Foundation grant \#584738}

\def\mytilde{\kern-.015in\hbox{\lower.03in\hbox{\~{}}}\kern-.01in}

\makeatletter

\let\choose\@@choose

\makeatother

\usepackage[left=1in, right=1in, top=1in, bottom=1in]{geometry}

\begin{document}
\begin{abstract} 
In this paper we study the relationship between three combinatorial formulas for type $A$ spherical Whittaker functions. In arbitrary type, these are spherical functions on $p$-adic groups, which arise in the theory of automorphic forms; they depend on a parameter $t$, and are specializations of Macdonald polynomials. There are three types of formulas for these polynomials, of which the first works in arbitrary type, while the other two in type $A$ only. The first formula is in terms of so-called alcove walks, and is derived from the Ram-Yip formula for Macdonald polynomials. The second one is in terms of certain fillings of Young diagrams, and is a version of the Haglund-Haiman-Loehr formula for Macdonald polynomials. The third formula is due to Tokuyama, and is in terms of the classical semistandard Young tableaux. We study the way in which the last two formulas are obtained from the previous ones by combining terms $-$ a phenomenon called compression. No such results existed in the case of Whittaker functions.
\end{abstract}

\maketitle

\section{Introduction}
\indent Whittaker functions, which arise in number theory, e.g. in the construction of $ L$-functions, are a basic tool in the theory of automorphic forms \cite{BBF,CL4}. They have numerous applications to Weyl group multiple Dirichlet series, combinatorial representation theory, and Schubert calculus on flag varieties \cite{BBL}.
\newline\indent {\em Spherical Whittaker functions} on $p$-adic groups $G$ are associated with an unramified principal series representation of ${G}$, and are defined by integrating over $U^-$ a spherical vector in this representation. More generally, the {\em Iwahori-Whittaker functions}, which are non-symmetric versions of spherical ones, are defined by integrating elements of a natural basis of Iwahori-fixed vectors in the principal series representation. Given $w$ in the Weyl group $W$, they can be viewed as functions  ${\mathcal W}_{\lambda,w}$ from weights $\lambda$ to the field of fractions $\mathcal{K}$ of Laurent polynomials on the weight lattice, which depend on a parameter $t$. Recently, deep connections of Whittaker functions to combinatorial representation theory and Schubert calculus on flag varieties were discovered.
\newline\indent The {\em Casselman-Shalika formula} \cite{casups,shiefc} expresses the spherical Whittaker function, for a dominant weight $\lambda$, as the corresponding irreducible character times a scaling factor in $\mathcal{K}$. Upon specializing to type $A$, the mentioned product becomes a $t$-deformation of the Vandermonde determinant times the Schur polynomial for the given dominant weight $\lambda$ \cite[Chapter V]{BBF}. In the 1980's it was discovered that this product can be expanded as a sum over certain combinatorial objects; this identity is known as \emph{Tokuyama's formula} \cite{G}. The combinatorial objects in question are known as Gelfand-Tsetlin patterns and are in bijection with SSYT \cite{BBF}. Tokuyama's formula plays a similar role to Macdonald's formula for Hall-Littlewood polynomials~\cite{IG}.
\newline\indent The Iwahori-Whittaker functions  were exhibited as specializations at $q\to \infty$ of non-symmetric Macdonald polynomials by Brubaker, Bump, and Licata \cite{BBL}.  Orr and Shimozono derived the corresponding specialization of the Ram-Yip formula for Macdonald polynomials, which is based on so-called alcove walks  \cite{OS}. We then consider a version  of the Haglund-Haiman-Loehr (HHL) formula~\cite{HHL0} for Macdonald polynomials corresponding to the type $A$ spherical Whittaker functions, which is in terms of certain fillings of Young diagrams. We show that we can derive such a formula using a function called \emph{fill} from alcove walks to fillings. More precisely, by summing over the preimage of this function we witness a phenomenon known as \emph{compression}. This leads to a further question: can we compress from this HHL-type formula to Tokuyama's formula using another function similar to the {\emph{fill} map?
\newline\indent Rigorously, we can define compression as follows. Let $F$ be a function such that we have two summation  representations of $F$ over sets $A$ and $B$ of combinatorial objects, namely
\begin{equation*}
F(x)=\sum_{a\in A}x(a)
\end{equation*}
\begin{equation*}
=\sum_{b\in B}y(b).
\end{equation*}
We say the bottom formula is a compressed form of the top formula if there exists a surjection $f:A\longrightarrow B$ such that 
\[
y(b)=\sum_{a \in f^{-1}(b)}x(a).
\]
\newline\indent As it turns out, the function needed to compress the HHL-type formula to Tokuyama's formula is simply sorting all the columns of an HHL-type filling in an increasing order. Its domain is HHL-type fillings, and we show that the image is the set of all SSYT of the given shape. Furthermore, by grouping all HHL-type fillings in the preimage of a given SSYT with respect to this ``sort'' map, we exhibit compression. We accomplish this by constructing an algorithm for generating the preimage of a SSYT, based on a binary search tree. In this way, we produce new statistics defined on a canonical HHL-type filling. On the other hand, the terms in Tokuyama's formula are computed using decorations on a SSYT known as \emph{separating walls} \cite{AP}. We show that these statistics are equivalent.    
\newline\indent Here is a schematic overview of the work done in the compression of the spherical Whittaker function.
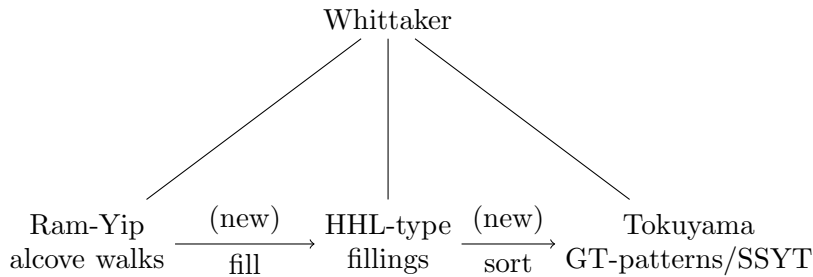
\[
\begin{tikzpicture}[node distance=3cm]
\tikzset{edge/.style = {->}}
\node(0){\text{Whittaker}};
\node(1)[below of =0, xshift=-4cm]{$\begin{matrix}\text{Ram-Yip} \\\text{alcove walks}\end{matrix}$};
\node(2)[below of =0]{$\begin{matrix}\text{HHL-type} \\\text{fillings}\end{matrix}$};
\node(3)[below of =0, xshift=4cm]{$\begin{matrix}\text{Tokuyama} \\\text{GT-patterns/SSYT}\end{matrix}$};
\draw(0)--(1);
\draw(0)--(2);
\draw(0)--(3);
\draw[edge] (1) to node[below]{fill} node[above]{(new)} (2) ;
\draw[edge] (2) to node[anchor=north]{sort}node[above]{(new)}(3) ;
\end{tikzpicture}
\]
\indent The compression phenomenon we exhibit in this paper is important for multiple reasons. First, the work done in type $A$ clarifies the relationship between the conceptual Ram-Yip formula and the more efficient and explicit formulas we present here. Indeed, the statistics on fillings would not have a good explanation unless one shows how they follow from the Ram-Yip formula. Second, compression offers tools for deriving efficient formulas in other Lie types, or even other specializations of Macdonald polynomials, where we have a Ram-Yip formula, but no formulas in terms of fillings of a Young diagram. Or, even if we do have the latter, we want to see if they fit into a general pattern. 

\indent We are working on a similar hierarchy of formulas for Whittaker functions of other classical types, i.e., from a Ram-Yip formula to an HHL-type formula, and then to a Tokuyama-type formula. Note that the second formula is not yet known, while it is not clear whether an existing Tokuyama type formula (such as that of Hamel-King, \cite{HK}) arises in this way. Furthermore, we expect our techniques to lead to a shorter and elementary derivation, via compression, of Macdonald's formula (in terms of semistandard Young tableaux) for the type $A$ Hall-Littlewood polynomials; by comparison, Klostermann's parallel approach \cite{IK} uses more involved concepts (such as ``one-skeleton galleries" and Bruhat-Tits buildings).

\indent This paper also opens a new line of research, related to metaplectic Whittaker functions, which correspond to an $n$-fold metaplectic cover of a $p$-adic group. A Ram-Yip type formula for the type $A$ metaplectic Whittaker functions was recently given in \cite{SA}. On the other hand, a Tokuyama-type formula for these functions (in terms of the corresponding highest weight crystals) was given in \cite{BBF}, and later reproved in different ways in \cite{McN} and \cite{AP1}. We expect that a compression phenomenon similar to the one studied here would explain the relationship between the two formulas. Furthermore, we expect such a result to extend to other classical types.

\section{Background}
\subsection{Spherical Whittaker functions}
\indent We have a combinatorial description of the type $A$ spherical Whittaker function known as Tokuyama's formula. This is expressed as a sum whose terms are computed via statistics on Gelfand-Tsetlin patterns. On another hand, based on the Casselman-Shalika formula, the same function is expressed as a Schur polynomial multiplied by a $t$-deformation of the Vandermonde determinant.
\subsection{Macdonald specializations} 
Macdonald polynomials are defined for any finite root system and are orthogonal with respect to a certain scalar product. Their coefficients are rational functions in $q$ and $t$. There are two types of Macdonald polynomials: the symmetric ones (i.e. invariant under the Weyl group action) and the non-symmetric ones. We denote the symmetric Macdonald polynomial by $P_{\lambda}(x;q,t)$ for $\lambda$ a dominant weight, and the non-symmetric one by $E_{\mu}(x;q,t)$ for $\mu$ an arbitrary weight. The symmetric Macdonald polynomials are obtained from the non-symmetric ones via a symmetrizing operator \cite[Chapter 5]{IG1}. 
\newline\indent In \cite{BBL}, it is shown that the Iwahori-Whittaker functions are specializations of non-symmetric Macdonald polynomials at $q\rightarrow \infty$:
\[
\mathcal{W}_{w,\lambda}(x)
=t^{\langle\lambda,\rho\rangle}(-t)^{\ell(w)}x^{-\rho}w_{0}E_{w_{0}w(\lambda+\rho)}(0,t^{-1});
\]
here $w_{0}$ is the longest element in the corresponding Weyl group. 
\subsection{Alcove path model}
We now introduce the alcove path model to derive combinatorial formulas for the spherical Whittaker function. There is the Ram-Yip formula for Macdonald polynomials of arbitrary type in terms of so-called alcove paths. This combinatorial model and the specialization of the Ram-Yip formula for spherical Whittaker functions to type $A$ is introduced. 
\newline\indent Let $\frak{g}$ be a complex Lie algebra of a connected, simply connected, simple complex Lie group. Let $\frak{h}$ be the Cartan subalgebra, and let $\Phi \subset \frak{h}^{*}$ be the corresponding irreducible root system. Let $\Phi^{+}$ be the set of positive roots and let $\alpha_{1},\ldots \alpha_{r}$ be the corresponding simple roots. Let $\langle\cdot,\cdot\rangle$ be the nondegenerate scalar product on $\frak{h}^{*}_{\mathbb{R}}$ induced by the Killing form. Given a root $\alpha$, define the corresponding coroot as $\alpha^{\vee}:=2\alpha/\langle\alpha,\alpha\rangle$. The Weyl group $W$ of the Lie group $G$ is generated by reflections $s_{\alpha}:\lambda\mapsto\lambda-\langle\lambda,\alpha^{\vee}\rangle\alpha$. If $w$ is a Weyl group element then it has a \emph{reduced decomposition} $w=s_{i_{1}}\ldots s_{i_{l}}$. We have the corresponding \emph{length function}, defined to be $\ell(w)=l$. Let $w_{\circ}$ be the unique longest element with $\ell(w_{\circ})=\#\Phi^{+}$. We say that $u$ \emph{covers} $w$ and write $u\gtrdot w$, if $w=us_{\beta}$, for some $\beta\in \Phi^{+}$ and $\ell(u)=\ell(w)+1$. By taking the transitive closure $>$ of the relation $\gtrdot$ we get the \emph{Bruhat order} on $W$. 
\newline \indent The \emph{weight lattice} $\Lambda$ is given by 
\[
\Lambda:=\{\lambda\in\frak{h}^{*}_{\mathbb{R}}| \text{ } \langle\lambda,\alpha^{\vee}\rangle\in \mathbb{Z} \text{ for any } \alpha \in \Phi^{+}\}
\] 
The weight lattice $\Lambda$ is generated by the fundamental weights $\omega_{1}, \ldots, \omega_{r}$ which are defined as the dual basis to the basis of simple coroots. The set $\Lambda^{+}$ of dominant weights is given by 
\[
\Lambda^{+}:=\{\lambda\in\Lambda| \text{ } \langle\lambda,\alpha^{\vee}\rangle\geq 0\text{ for any } \alpha \in \Phi^{+}\}
\]
Let $\rho:=\omega_{1}+\ldots+\omega_{r}=\frac{1}{2}\sum_{\beta\in\Phi^{+}}\beta$. 
\newline\indent We are now ready to define \emph{aloves}. Let $W_{\text{aff}}$ be the affine Weyl group of the Langlands' dual group $G^{\vee}$. The affine Weyl group $W_{\text{aff}}$ is generated by affine reflections $s_{\alpha,k}:\frak{h}^{*}_{\mathbb{R}}\longrightarrow \frak{h}^{*}_{\mathbb{R}}$, for $\alpha \in \Phi$ and $k\in\mathbb{Z}$, that reflect the space $\frak{h}^{*}_{\mathbb{R}}$ with respect to the affine hyperplanes 
\[
H_{\alpha,k}:=\{\lambda\in\frak{h}^{*}_{\mathbb{R}}| \text{ } \langle\lambda,\alpha^{\vee}\rangle=k\}.
\]
Explicitly the formula for such a reflection is 
\[
s_{\alpha,k}:\lambda\mapsto s_{\alpha}(\lambda)+k\alpha=\lambda-(\langle\lambda,\alpha^{\vee}\rangle-k)\alpha
\]
The hyperplanes $H_{\alpha,k}$ divide the real vector space $\frak{h}^{*}_{\mathbb{R}}$ into open regions called \emph{alcoves}. Each alcove $A$ is given by inequalities of the form 
\[
A:=\{\lambda \in \frak{h}^{*}_{\mathbb{R}}|\text{ } m_{\alpha}<\langle\lambda,\alpha^{\vee}\rangle<m_{\alpha}+1 \text{ for all } \alpha \in \Phi^{+}\}
\]
where $m_{\alpha}=m_{\alpha}(A)$, $\alpha \in \Phi^{+}$, are some integers. Note that the affine Weyl group acts simply transitively on the collection of all alcoves. The \emph{fundamental alcove} $A_{\circ}$ is given by 
\[
A_{\circ}:=\{\lambda\in\frak{h}^{*}_{\mathbb{R}}|\text{ } 0<\langle\lambda,\alpha^{\vee}\rangle<1\text{ for all } \alpha \in \Phi^{+}\}
\]
There is a one-to-one correspondence between affine Weyl group elements and alcoves.
\newline\indent We say that two alcoves are \emph{adjacent} if they are distinct and have a common wall. Given a pair of adjacent alcoves $A$ and $B$, we write $A\xrightarrow[]{\beta}B$ if the common wall is of the form $H_{\beta,k}$ and the root $\beta\in\Phi$ points in the direction from $A$ to $B$.
\begin{definition} An \emph{alcove path} is a sequence of alcoves $(A_0, A_1, \ldots, A_m)$ such that $A_{j-1}$ and $A_j$ are adjacent, for $j=1,\ldots, m$. We say that an alcove path is \emph{reduced} if it has minimal length among all alcove paths from $A_0$ to $A_m$.
\end{definition}
Let $A_{\lambda}=A_{\circ}+\lambda$ be the translation of the fundamental alcove $A_{\circ}$ by the weight $\lambda$.
\begin{definition}
The sequence of roots $(\beta_1,\beta_2,\ldots,\beta_m)$ is called a \emph{$\lambda$-chain} if 
\[
\begin{matrix}
A_{\circ}&\xrightarrow[]{\beta_1}& A_1&\xrightarrow[]{\beta_2}& \ldots &\xrightarrow[]{\beta_m} & A_m
\end{matrix}
\]
\end{definition}
We now fix a dominant weight $\lambda$ and an alcove path $\Pi=(A_0,\ldots,A_m)$ from $A_0=A_{\circ}$ to $A_m=A_{-\lambda}$ \cite{CL7}.
\subsection{Specializing to Type A Root System} We now focus our attention on the type $A_{n-1}$ root system. The Weyl group $W$ is the symmetric group $S_{n}$. We use one-line notation for such permutations, i.e. $w=w(1)\ldots w(n)$. We can identify the dual of the Cartan subalgebra over the reals, $\frak{h}^{*}_{\mathbb{R}}$ with the quotient space $V=\mathbb{R}^{n}/\mathbb{R}(1,\ldots ,1)$ where $\mathbb{R}(1,\ldots1)$ denotes the one-dimensional subspace of $\mathbb{R}^{n}$ spanned by the vector $(1,\ldots ,1)$. The symmetric group acts on this quotient space by permuting the coordinate vectors in $\mathbb{R}^{n}$. Let $\epsilon_{1},\ldots , \epsilon_{n}\in V$ be the images of the standard basis vectors of $\mathbb{R}^{n}$ under this quotient. The root system $\Phi$ can be represented as $\Phi=\{\alpha_{ij}=\epsilon_{i}-\epsilon_{j}:\text{ }i\neq j, \text{ } 1\leq i,j \leq n\}$. The simple roots are $\alpha_{i}=\epsilon_{i}-\epsilon_{i+1}$ for $i=1, \ldots , n-1$. The fundamental weights are $\omega_{i}=\epsilon_{1}+\ldots+\epsilon_{i}$ for $i=1,\ldots , n-1$. The weight lattice is $\Lambda=\mathbb{Z}^{n}/\mathbb{Z}(1,\ldots , 1)$. A dominant weight $\lambda=\lambda_{1}\epsilon_{1}+\ldots+\lambda_{n-1}\epsilon_{n-1}$ is identified with a partition $(\lambda_{1}\geq \lambda_{2} \geq \ldots \geq\lambda_{n-1}\geq \lambda_{n}=0)$ of length at most $n-1$. We fix a partition $\lambda$. For simplicity we denote both the root, $\alpha_{ij}$ and the reflection $s_{\alpha_{ij}}$ (which is the transposition of values $i$ and $j$) by $(i,j)$.
\newline\indent
We define a chain of roots which will be used in the alcove formula for spherical Whittaker functions. Consider a chain of roots denoted $\overline{\Gamma}(k)$ given by: 
\[
\begin{matrix} 
((1,k+1), & (1,k+2), & \ldots, & (1,n),\\
			(2,k+1),& (2,k+2), & \ldots, & (2,n)\\
			& &\ldots &\\
			(k,k+1),& (k,k+2), & \ldots, & (k,n))
\end{matrix}
\]
Denote by $\overline{\Gamma}'(k)$ the chain of roots obtained by removing the root $(i,k+1)$ at the beginning of each row. Now define a $\lambda$-chain $\overline{\Gamma}$ as a concatenation $\overline{\Gamma}:=\overline{\Gamma}_{1}\ldots \overline{\Gamma}_{\lambda_{1}}$, where 
\begin{equation}
\overline{\Gamma}_{j}:=
\begin{cases}
\overline{\Gamma}'(\lambda'_{j}) & \textrm{ if } j=\textrm{min}\{i:\lambda'_{i}=\lambda'_{j}\}\\
\overline{\Gamma}(\lambda'_{j}) & \textrm{ otherwise}
\end{cases}
\end{equation}
Furthermore, we define $\overline{\Gamma}(k,p)$ to be $\overline{\Gamma}(k)$ with the first $p$ entries removed.
\newline\indent For context, we begin with the alcove path from the fundamental alcove $A_{\circ}$ to $A_{\circ}-\lambda$ obtained by concatenating $\omega_{i}$-chains, similar to as in \cite[Section 3.1]{CL2}. We then augment this path by using Coxeter moves to bring reflections over hyperplanes which pass through the origin to the beginning of the path. If the weight $\lambda$ is regular, then this procedure guarantees that our alcove path, goes through $w_{\circ}A_{\circ}=-A_{\circ}$. We can then remove this section of the path and consider the above chain of roots, $\overline{\Gamma}$.

\subsection{Alcove formula for the spherical Whittaker function}
We are now ready to state the alcove formula for the spherical Whittaker function. Note that this formula can be written in any finite root system. Using the Macdonald specialization given in subsection 2.2, the Whittaker function is indexed by a weight $\lambda$. We then consider the weight $\lambda+\rho$ which is always regular. Therefore we can use the above chain of roots, $\overline{\Gamma}$, to express the spherical Whittaker function as a sum over increasing chains with respect to the Bruhat order. 
\begin{definition} Let $m$ be the number of transpositions in $\overline{\Gamma}$, and define
\begin{align*}
\mathcal{A}(\overline{\Gamma})=\{(u,K)\in W\times2^{[m]}:\text{ } K=\{k_{1},\ldots,k_{s}\}, u<ur_{k_{1}}<ur_{k_{1}}r_{k_{2}}<\ldots<ur_{k_{1}}\ldots r_{k_{s}}=w\},
\end{align*}
where $r_{k_{i}}$ is the $k_{i}$-th transposition in $\overline{\Gamma}$. In addition, we define $uK=w$ for some $w$ in the Weyl group.
\end{definition}
We have the following formula for spherical Whittaker functions which can be derived from \cite[Theorem 6.2]{OS}.:
\begin{equation}
\widetilde{\mathcal{W}}_{\lambda}=\sum_{(u,K)\in\mathcal{A}(\overline{\Gamma})}(-1)^{\ell(uK)}t^{\frac{1}{2}(\ell(u)+\ell(uK)-|K|)}(t-1)^{|K|}x^{-u\text{wt}(K)-\rho}.
\end{equation}
Note that this formula is in striking similarity to Ram's formulation of Schwer's formula for Hall-Littlewood polynomials \cite[Theorem 2.7]{CL1}.

\subsection{Tableaux formula for spherical Whittaker functions}
 We are now ready to introduce \emph{Tokuyama's formula}. This formula gives a description of the spherical Whittaker function in terms of concrete combinatorial objects known as Gelfand-Tsetlin (GT) patterns which are in bijection with semistandard Young tableaux. We begin with the definitions and statistics on such objects.
\begin{definition}\cite{G}
A Gelfand-Tsetlin (GT) pattern is a triangular array of nonnegative integers of the form 
\[
\begin{matrix}
a_{1,1}& & a_{1,2} & & a_{1,3} & & \ldots & & a_{1,n} \\
& a_{2,2}& & a_{2,3} & & \ldots& & a_{2,n}\\
& & \ldots& & \ldots & & \ldots\\
& & & a_{n-1,n-1} & & a_{n-1,n}\\
& & & & a_{n,n}
\end{matrix}
\]
where each row $r_{i}=(a_{i,i},a_{i,i+1},\ldots,a_{i,n})$ is a weakly decreasing partition, and any two consecutive rows $r_{i}$ and $r_{i+1}$ satisfies the interleaving condition:
\[
a_{i-1,j-1}\geq a_{i,j} \geq a_{i-1,j}.
\]
Furthermore, a \emph{strict} GT pattern is one in which each row, $r_{i}$ is strictly decreasing. Denote the set of all strict GT patterns with top row $\alpha$ as $SGT(\alpha)$. 
\end{definition} 
We have the following statistics on GT patterns:
\begin{definition} An entry $a_{i,j}$ in a GT pattern is 
\begin{itemize}
\item \emph{left-leaning} if $a_{i,j}=a_{i-1,j-1}$
\item \emph{right-leaning} if $a_{i,j}=a_{i-1,j}$
\item \emph{special} if it is neither left-leaning nor right-leaning. 
\end{itemize}
denote the number of left-leaning, right-leaning, and special entries by $l(T)$, $r(T)$ and $z(T)$ respectively, for a GT pattern $T$.  
\end{definition}
Given a GT pattern, T with $n$ rows, we define
\[
m_{i}(T)=\begin{cases}
|r_{i}|-|r_{i+1}| & \text{ for } 1\leq i \leq n-1 \\
|r_{i}| & \text{ for } i=n
\end{cases}
\]
Finally, define
\[
m(T)=(m_{1}(T),\ldots, m_{n}(T)).
\]
We are now ready to state Tokuyama's formula given as the following theorem. 
\begin{theorem}\cite{G}
For any weakly decreasing partition $\lambda$ of length $n$, we have 
\[
\widetilde{\mathcal{W}}_{\lambda}=\sum_{T\in SGT(\lambda+\rho)}(1-t)^{z(T)}(-t)^{l(T)}x^{m(T)}
\]
where SGT denote strict Gelfand-Tsetlin patterns.
\end{theorem}
\indent We introduce the following bijection between GT patterns and SSYT in order to describe Tokuyama's formula in terms of SSYT. This is advantageous for compression because previous work in this area has only dealt with fillings of Young diagrams. So we need to translate this formula to accomplish the compression from the alcove formula for spherical Whittaker functions to Tokuyama's formula. 
\newline\indent There is a classical bijection between GT patterns and semistandard Young tableau given by consecutively deleting all boxes with the same entry starting from largest to smallest. This procedure gives a sequence of partitions that make up the triangular array. To go back to semistandard Young tableau, we simply take the difference of the partitions in the rows of the GT-pattern, giving a row strip which can then then be filled with the entry $n-i+1$ where $i$ is the row of the GT pattern. Putting these row strips together, justified to the left, gives a semistandard Young tableau. These statistics on GT patterns can be translated to the following statistics on SSYT via this bijection. We now define decorations on SSYT known as separating walls and their associated statistics which are in correspondence with the statistics on GT patterns which can be found in \cite[Chapter 3]{AP}.
\begin{definition}
Let $T$ be a SSYT. A \emph{separating wall of index k}, denoted $|_{k}$ is a decoration on $T$, placed after the last box of T with entry less than or equal to $k$ and before the first box with entry greater than $k$ with respect to the reading order; walls may be placed at the beginning or end of a row. 
\end{definition}
When placing separating walls, read the tableau from left to right, and top to bottom. We only place separating walls of index greater than or equal to the row index. It will be relevant to include a separating wall of index equal to $\#\textrm{rows}+1$, below the last row, justified to the left. 
\begin{example}
\emph{Let} \[ T=\young(1112223447,223356,4555).\]
\emph{After placing the separating walls of index less than or equal to 7, we have:}
\begin{figure}[h]
$$T=\mathrlap{\Yboxdim27pt\young(1112223447,223366,4555)}\begin{matrix}\\ \\ \kern2.4cm\textcolor{red}{\Big|_{1}}& \kern2.2cm \textcolor{red}{\Big|_{2}}& \kern.3cm \textcolor{red}{\Big|_{3}}&\kern1.2cm \textcolor{red}{\Big|_{4}}\textcolor{red}{\Big|_{5}}\textcolor{red}{\Big|_{6}} & \kern-.2cm\textcolor{red}{\Big|_{7}}\\ & &&& \\ \kern.5cm\textcolor{red}{\Big|_{2}}& \kern-1cm\textcolor{red}{\Big|_{3}}\textcolor{red}{\Big|_{4}}\textcolor{red}{\Big|_{5}}&&\kern-4.5cm\textcolor{red}{\Big|_{6}}\textcolor{red}{\Big|_{7}}&\\  \\\kern-2.9cm\textcolor{red}{\Big|_{3}} & \kern-7.2cm\textcolor{red}{\Big|_{4}}&\kern-4.7cm \textcolor{red}{\Big|_{5}}\textcolor{red}{\Big|_{6}}\textcolor{red}{\Big|_{7}}\\ \kern-2.9cm\textcolor{red}{\Big|_{4}}\end{matrix}$$
\end{figure}
\end{example}
We use the term \emph{strictly to the right (resp. left)} to indicate the presence of a box, i.e. there exists an entry in $T$ that divides two spearating walls. We will also use the term \emph{directly above} (resp. \emph{below)} to indicate the position of a separating wall exactly one row above (resp. below) another wall.
\begin{definition}
The \emph{boxed, circled statistic} on a SSYT, $T$, denoted $n(T)$ is the number of pairs of separating walls $|_{k-1}|_{k}$, such that there is no entry between them and they are directly above a pair of separating walls $|_{k}|_{k+1}$ with no entry between them.
\end{definition}
\begin{definition}
A SSYT, $T$, is \emph{strict} if the corresponding GT pattern is strict. 
\end{definition}
Note that $n(T)>0$ if and only if $T$ is not strict. 
\begin{definition}
The \emph{boxed, not circled statistic} on a SSYT, $T$,  denoted, $l(T)$ is the number of separating walls $|_{k}$ such that $|_{k}$ is directly above $|_{k+1}$ in $T$.
\end{definition}
\begin{definition}
The \emph{not boxed, not circled statistic} on a SSYT, $T$, denoted, $z(T)$ is the number of separating walls $|_{k}$ such that there exists two separating walls $|_{k+1}$ and $ |'_{k+1}$, under the condition that $|_{k+1}$ is in the same row and strictly to the right of $|_{k}$ and $|'_{k+1}$ is directly below and strictly left of $|_{k}$.
\end{definition}
In the previous example we have $n(T)=1$, $l(T)=2$, and $z(T)=3$.

\section{Whittaker Weak Compression}
In this section we review what we know for the spherical Whittaker function. There are two formulas of relevance when studying the spherical Whittaker function, the Ram-Yip formula derived in \cite{BBL} and \cite{OS}, and Tokuyama's formula derived in the 1980's and stated in \cite{G} and \cite[Chapter V]{BBF}. The Ram-Yip formula for the spherical Whittaker function has a $x^{\rho}$ monomial factor. This corresponds to considering Young diagrams of shape $\lambda+\rho$ where $\rho=(n-1,n-2,\ldots,2,1)$. This causes our fillings to have $n-1$ distinct parts, an assumption that was necessary in \cite{CL1}.
\newline\indent Next we define the necessary data used in expressing the compression phenomenon. We denote the conjugate partition by $\lambda'$ (i.e. the reflection of the diagram of $\lambda$ in the line $y=-x$). For any cell $u=(i,j)$ of $\lambda$, denote the cell $v=(i,j+1)$ directly to the right of $u$ by $\textrm{r}(u)$.
\newline\indent Two cells $u,v \in \lambda$ are said to \emph{attack} each other if either
\begin{enumerate}\setcounter{enumiii}{1}
\item they are in the same column: $u=(i,j),$ $v=(k,j)$; or
\item they are in consecutive columns, with the cell in the left column strictly below the one in the right column: $u=(i,j),$ $v=(k,j+1),$ where $i>k$.
\end{enumerate}

The figure below shows the two types of pairs of attacking cells.
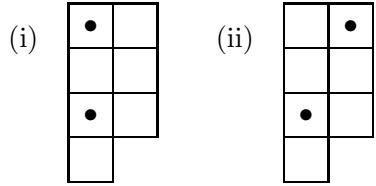
\begin{equation*}
\textrm{(i)\quad} \ydiagram[\bullet]{1,0,1}*[*(white)]{2,2,2,1} \textrm{\quad \quad}
\textrm{(ii)\quad} \ydiagram[\bullet]{1+1,0,1}*[]{2,2,2,1}
\end{equation*}

\indent A \emph{filling} is a function $\sigma:\lambda \rightarrow [n]:=\{1,\ldots, n\}$ for some $n$. This is a function which is defined on the cells of $\lambda$ which assigns for each cell a value in $[n]$. We define the content of a filling $\sigma$ as $\textrm{ct}(\sigma):=(c_{1}, \ldots, c_{n})$, where $c_{i}$ is the number of entries $i$ in the filling, i.e., $c_{i}:=|\sigma^{-1}(i)|.$ The monomial $x^{\textrm{ct}(\sigma)}$ of degree $m$ in the variables $x_{1},\ldots x_{n}$ is then given by $$x^{\textrm{ct}(\sigma)}:=x_{1}^{c_{1}}\ldots x_{n}^{c_{n}}.$$

\begin{definition}\label{lst:1}
Let $F(\lambda, n)$ denote the set of fillings $\sigma: \lambda \rightarrow [n]$ satisfying
\begin{itemize}
\item $\sigma(u) \neq \sigma(v)$ whenever $u$ and $v$ attack each other, and
\item $\sigma$ is weakly increasing in rows, i.e. $\sigma(u) \leq \sigma(\textrm{r}(u)).$
\end{itemize}
We call such fillings \emph{Haglund-Haiman-Loehr-type fillings} or \emph{HHL} for short. 
\end{definition}

\indent The \emph{reading order} is the total order on the cells of $\lambda$ given by reading each column from top to bottom, and by considering the columns from left to right.
\begin{definition}
The \emph{inversion statistic} on a filling $\sigma$, denoted $\emph{inv}(\sigma)$ is the number of pairs of cells $(u,v)$ with $v=(i,k+1)$ and $u=(j,k)$ where $i<j$ and  $\sigma(u)<\sigma(v)<\sigma(w)$, where $w$ is the cell directly to the right of $u$, if it exists (otherwise, the condition consists only of the first inequality). The \textrm{descent statistic}, denoted $\emph{des}(\sigma)$, is the number of cells $u=(i,j)$ with $\sigma(u) < \sigma (\textrm{r}(u)).$
\end{definition}

We now define the filling map. Given a pair $(w,T)\in \mathcal{A}(\overline{\Gamma})$, we consider the permutations $$ \pi_{j}=\pi_{j}(w,T):=wT_{1}\ldots T_{j}$$ for $j=1,\ldots \lambda_{1}$. In particular $\pi_{\lambda_{1}}=wT$

\begin{definition}
The \emph{filling map} is the map $f$ which maps a pair $(w,T)$ to a filling $\sigma$ of shape $\lambda$, defined by $$\sigma(i,j):=\pi_{j}(i).$$ In other words, the $j$-th column of the filling $\sigma$ consists of the first $\lambda'_{j}$ entries of the permutation $\pi_{j}$. 
\end{definition}
It was shown in \cite[Proposition 3.6]{CL1} that the content of the image of a pair under the fill map is equal to a permutation of the weight determined by the sequence of roots $T$, so $\textrm{ct}(f(w,T))=w(\mu(T))$.  
\begin{theorem}\label{whittaker}
Given a dominant weight $\lambda$ of size $n$, and $\sigma \in F(\lambda+\rho, n)$. We have \begin{equation*}
\sum_{(w,T) \in f^{-1}(\sigma)} (-1)^{\ell(wT)}t^{\frac{1}{2}(\ell(w)+\ell(wT)-|T|)}(t-1)^{|T|}=(-1)^{n-a_{0}+\ell(C)}t^{n-a_{0}+\emph{inv}(\sigma)}(1-t)^{\emph{des}(\sigma)}.
\end{equation*}
Where $C$ is the first column of $\sigma$, and $a_{0}$ is the unique entry missing from $C$.
\end{theorem}
After summing over all monomials $x^{\text{ct}(\sigma)}$ we have the following formula for the spherical Whittaker function. 
\begin{theorem}
Let $\rho$ be the partition $(n-1,\ldots,1)$. For any partition $\lambda$, we have
\begin{equation*}
\widetilde{W}_{\lambda}=\sum_{\sigma \in F(\lambda+\rho, n)}(-1)^{n-a_{0}+\ell(C)}t^{n-a_{0}+\emph{inv}(\sigma)}(1-t)^{\emph{des}(\sigma)}x^{\emph{ct}(\sigma)},
\end{equation*}
where $C$ is the first column of $\sigma$ and $a_{0}$ is the unique element absent from $C$.
\end{theorem}
Now we outline the idea of the Whittaker weak compression. We first fix $\sigma$ an HHL-type filling. We consider a slightly different formula which we can manipulate algebraically to yield our original formula. We fix an initial permutation $w$ in our pair, so the sum is over $T$. The first step in the compression is to reduce the problem to a single entry in a column while considering the subset of roots in our chain which give this entry. We then extend this to an entire column. After this, we consider pairs of columns where the left column is the starting column and the right column is the terminating column and consider sequences of roots which perform this transformation. We finally extend this to the entire filling. To prove the theorem, we consider the sum over all pairs in the preimage of $\sigma$ and use a technique similar to one found in \cite{CL1}. Here is the sequence of propositions needed to prove the main theorem. 
\begin{center}
\begin{tikzpicture}[node distance=0cm]
\node(2)[xshift=6cm]{\text{Theorem 3.1}};
\node(3)[xshift=-3cm]{\text{Prop 3.1}};
\node(4)[xshift=0]{\text{Prop 3.2}};
\node(5)[xshift=3cm]{\text{Prop 3.3}};
\draw[->] (3) edge (4);
\draw[->] (4) edge (5);
\draw[->] (5) edge (2);
\end{tikzpicture}
\end{center}
\subsection{Proof of compression}
Let $\lambda$ be an arbitrary weight/partition. Given a sequence of positive integers, $w$, we write $w[i,j]$ for the subsequence $w(i)w(i+1) \ldots w(j)$. We use the notation $N_{a}(w)$ and $N_{ab}(w)$ for the number of entries $w(i)$ with $w(i)<a$ and $a<w(i)<b$, respectively. Given a sequence of transpositions 
$T=((a_{1},b_{1}), \ldots ,(a_{p},b_{p}))$ and a permutation $w$, we define 
$$
N(w,T)=\sum_{i=1}^{p} N_{c_{i}d_{i}}(w_{i}[a_{i},b_{i}])
$$
where
$$
w_{i}=w(a_{1},b_{1}), \ldots , (a_{i},b_{i}), \quad c_{i}:=min(w_{i}(a_{i}),w_{i}(b_{i})),   \quad d_{i}:=max(w_{i}(a_{i}),w_{i}(b_{i})).$$

We will need a series of propositions first.
\begin{prop}\label{prop1}
Consider a permutation $w \in S_{n}$ and a number $b \in [n]\setminus \{a\}$, where $a:=w(1)$. Consider an integer $p$ such that $0\leq p < w^{-1}(b)-1$. Then we have 
\begin{equation} \label{eq:10}
\sum_{\substack{T:(w,T)\in \mathcal{A}(\overline{\Gamma}(1,p)) \\ wT(1)=b}} t^{N(w,T)}(1-t)^{|T|}=t^{N_{a,b}(w[2,p+1])}(1-t)^{1- \delta_{ab}}.
\end{equation} 
\end{prop}
\begin{proof}
Assume $a<b$. We use decreasing induction on $p$. Let $p=w^{-1}(b)-2$. Since $\overline{\Gamma}(1,p)=((1,w^{-1}(b)), \ldots , (1,n))$, then we have a single term on the left, which gives, $$t^{N(w,T)}(1-t)^{|T|}=t^{N_{a,b}(w[2,p+1])}(1-t).$$ We now prove the statement for $p$ assuming it for $p+1$. If $c:=w(p+2)\not \in (a,b)$, the induction is straightforward. If $c<a$ or $c>b$ then $(1,p+2)\not \in T$ since $wT(1)=b$ and we have the increasing chain condition on $\mathcal{A}(\overline{\Gamma}(1,p))$. So the induction hypothesis carries through in this case. Now, assume the contrary. Denote the sum on the left hand side of (\ref{eq:10}) by $S(w,p)$. $S(w,p)$ splits into two sums depending on if $(1,p+2)\not \in T$ and $(1,p+2)\in T$. By induction the first sum is $$S(w,p+1)=t^{N_{a,b}(w[2,p+1])+1}(1-t).$$ The second sum is \begin{align*}t^{N_{a,c}(w[2,p+1])}(1-t)S(w(1,p+2),p+1)=\\t^{N_{a,c}(w[2,p+1])}(1-t)t^{N_{b,c}(w[2,p+1])}(1-t)=\\t^{N_{a,b}(w[2,p+1])}(1-t)^{2}.\end{align*} The desired result follows by adding the two sums into which $S(w,p)$ splits.
\end{proof}

\begin{prop}\label{prop2} Consider two sequences $C_{1} \leq C_{2}$ of size $k$ and entries in $[n]$, where $\leq$ denotes the componentwise order. Also consider a permutation $w\in S_{n}$ such that $w[1,k]=C_{1}$. Then we have 
\begin{equation}\label{eq:2}
\sum_{\substack{T:(w,T)\in \mathcal{A}(\overline{\Gamma}(k)) \\ wT[1,k]=C_{2}}} t^{N(w,T)}(1-t)^{|T|}=t^{\ell(C_{2})-(\ell(C_{1})-\emph{inv}(C_{1}C_{2}))}(1-t)^{\emph{des}(C_{1}C_{2})}.
\end{equation}
Where $C_{1}C_{2}$ denotes the two column filling with left column $C_{1}$ and right column $C_{2}$.
\end{prop}
\begin{proof}
Split $\overline{\Gamma}(k)$ as $\overline{\Gamma}_{1}(k)\ldots \overline{\Gamma}_{k}(k)$ by rows, so$$\overline{\Gamma}_{i}(k)=((i,k+1), (i,k+2), \ldots, (i,n)).$$ This splitting induces one for the subsequence $T$ of $\overline{\Gamma}(k)$, namely $T=T_{1}\ldots T_{k}$. For $i=0,\ldots, k$ define $w_{i}=wT_{1}\ldots T_{i}$, in particular $w_{0}=w$. So we can split the sum in (\ref{eq:2}) as a $k$-fold sum as follows
\begin{equation*}
\sum_{\substack{T_{1}:(w,T_{1})\in \mathcal{A}(\overline{\Gamma}_{1}(k)) \\ w_{0}T_{1}(1)=C_{2}(1)}}t^{N(w_{0},T_{1})}(1-t)^{|T_{1}|}\ldots \sum_{\substack{T_{k}:(w_{k-1},T_{k})\in A(\overline{\Gamma}_{k}(k)) \\ w_{k-1}T_{k}(k)=C_{2}(k)}}t^{N(w_{k-1},T_{k})}(1-t)^{|T_{k}|}
\end{equation*}
By Proposition \ref{prop1}, we have 
\begin{equation*}
\sum_{\substack{T_{i}:(w_{i-1},T_{i})\in \mathcal{A}(\overline{\Gamma}_{i}(k)) \\ w_{i-1}T_{i}(i)=C_{2}(i)}}t^{N(w_{i-1},T_{i})}(1-t)^{|T_{i}|}=t^{N_{C_{1}(i),C_{2}(i)}(C_{1}[i+1,k])}(1-t)^{1- \delta_{C_{1}(i)C_{2}(i)}}.
\end{equation*}
So we get $t^{e}(1-t)^{des(C_{1}C_{2})}$ where $$e=\sum_{i=1}^{k-1}N_{C_{1}(i),C_{2}(i)}(C_{1}[i+1,k]).$$
Clearly, $e$ is the number of ascents $(i,j)$ in $C_{1}$(meaning that $i<j$ and $C_{1}(i)<C_{1}(j)$) for which $C_{1}(j)<C_{2}(i)$. If $(i,j)$ is an ascent in $C_{1}$ not satisfying the previous condition, so $C_{1}(j)> C_{2}(i)$, then $(i,j)$ is an ascent in $C_{2}$ by the weakly increasing rows condition on fillings. The only ascents in $C_{2}$ which do not arise this way are counted by $\textrm{inv}(C_{1}C_{2})$. Therefore, $e=\binom{k}{2}-\ell(C_{1})-(\binom{k}{2}-\ell(C_{2})-\textrm{inv}(C_{1}C_{2}))=\ell(C_{2})-(\ell(C_{1})-\textrm{inv}(C_{1}C_{2}))$.
\end{proof}

\begin{prop}\label{prop3}
Consider a filling $\sigma \in F(\lambda, n)$, whose rightmost column in \newline $C:=(\sigma(1,1), \ldots, \sigma(\lambda'_{1},1))$ Let $w$ be a permutation in $S_{n}$ satisfying $w[1,\lambda'_{1}]=C$. We have 
\begin{equation}\label{eq:3}
\sum_{T:(w,T)\in f^{-1}(\sigma)}t^{N(w,T)}(1-t)^{|T|}=t^{\textrm{inv}(\sigma)-\ell(C)}(1-t)^{\textrm{des}(\sigma)}.
\end{equation}
\end{prop}
\begin{proof}
$\overline{\Gamma}=\overline{\Gamma}_{1}\ldots \overline{\Gamma}_{\lambda_{1}}$ induces a splitting in $T$, as $T_{1}\ldots T_{\lambda_{1}}$. Since $\lambda+\rho$ has $n-1$ boxes in its first column, $\overline{\Gamma}_{1}$ is empty, which implies $T_{1}$ is empty (by definition of $\overline{\Gamma}$). Let $m=\lambda_{1}$ be the number of columns of $\lambda$, and let $C=C_{1}, \ldots, C_{m}$ be the columns of $\sigma$, of lengths $c_{1}:=\lambda'_{1},\ldots, c_{m}:=\lambda'_{m}$. Let $C'_{i}:=c_{i}[1,c_{i+1}]$ for $i=1,\ldots,m-1$. For $i=1,\ldots,m$, define $w_{i}:=wT_{1}\ldots T_{i}$, so $w_{1}=w$. The sum in (\ref{eq:3}) can be written as an $m-1$-fold sum 
\begin{equation*}
\sum_{\substack{T_{2}:(w_{1},T_{2})\in A(\overline{\Gamma}_{2})\\ w_{2}[1,c_{2}]=C_{2}}}t^{N(w_{1},T_{2})}(1-t)^{|T_{2}|}\ldots \sum_{\substack{T_{m}:(w_{m-1},T_{m})\in A(\overline{\Gamma}_{m})\\w_{m}[1,c_{m}]=C_{m}}}t^{N(w_{m-1},T_{m})}(1-t)^{|T_{m}|}
\end{equation*}
By Proposition $\ref{prop2}$, we have 
\[
\sum_{\substack{T_{i}:(w_{i-1},T_{i})\in A(\overline{\Gamma}_{i})\\w_{i}[1,c_{i}]=C_{i}}}t^{N(w_{i-1},T_{i})}(1-t)^{|T_{i}|}=
\]
\[
 t^{\ell(C_{i})-(\ell(C'_{i-1})-\textrm{inv}(C'_{i-1}C_{i}))}(1-t)^{\textrm{des}(C_{i-1}C_{i})}
 \]

We can see that the above sum does not depend on the permutation $w_{i-1}$, and only on $C_{i-1}$ and $C_{i}$. Therefore, the $m-1$-fold sum is a product of $m-1$ factors, and evaluates to $t^{e}(1-t)^{\textrm{des}(\sigma)}$, where 
\begin{equation*}
e=\sum_{i=2}^{m}\ell(C_{i})-\ell(C'_{i-1})+\textrm{inv}(C'_{i-1}C_{i}).
\end{equation*}  
Thus,
\[e+\ell(C)=\ell(C_{1})+\sum_{i=2}^{m}\ell(C_{i})-\ell(C'_{i-1})+\textrm{inv}(C'_{i-1}C_{i})=
\]
\[
 \ell(C_{m})+\sum_{i=2}^{m}\ell(C_{i-1})-\ell(C'_{i-1})+\textrm{inv}(C'_{i-1}C_{i})=\textrm{inv}(\sigma).
\]
Where the last equality comes from the fact that $\ell(C_{i-1})-\ell(C'_{i-1})$ is the number of descents in $C_{i-1}$ which include the bottom entry of this column. Assume $c_{i-1}\not = c_{i}$ (otherwise, $\ell(C_{i-1})-\ell(C'_{i-1})=0$). Since we have weakly increasing rows in $\sigma$ then any descent in $C_{i-1}$ involving the bottom entry of this column is counted by $\textrm{inv}(\sigma)$ because it has no entry to its right. In addition, $\ell(C_{m})=0$ since $c_{m}=1$.
\end{proof}

We are now ready to prove Theorem $\ref{whittaker}$.

\begin{proof1}
We rewrite the sum in Theorem $\ref{whittaker}$, by splitting $T$ as $T_{1}T'$. We let $w'=wT_{1}$ and consider decreasing chains starting at $w'$.
\[
\sum_{(w,T)\in f^{-1}(\sigma)}(-1)^{\ell(wT)}t^{\frac{1}{2}(\ell(w)+\ell(wT)-|T|)}(t-1)^{|T|}=
\]
\[
\sum_{\substack{(w',T')\in f^{-1}(\sigma)\\ T'_{1}:(w',T^{r}_{1})\in \mathcal{A}'(\overline{\Gamma}^{r}(\lambda'_{1}))}}(-1)^{\ell(wT)}t^{\frac{1}{2}(\ell(w'T^{r}_{1})+\ell(w'T')-|T'|-|T_{1}|)}(t-1)^{|T_{1}|+|T'|}
\]
Next, we add and subtract $\frac{1}{2}\ell(w')$ in the exponent of $t$ and split the sum by terms depending on $T_{1}$, leading to the following product
\[
\sum_{(w',T')\in f^{-1}(\sigma)}(-1)^{\ell(w'T')}t^{\frac{1}{2}(\ell(w'T')-\ell(w')-|T'|)}(t-1)^{|T'|}\times
\]
\[
\sum_{T^{r}_{1}:(w',T^{r}_{1})\in A'(\overline{\Gamma}^{r}(\lambda'_{1}))}t^{\frac{1}{2}(\ell(w'T^{r}_{1})+\ell(w')-|T_{1}|)}(t-1)^{|T_{1}|}
\]
Since $T^{r}_{1}$ is always empty (by definition of $\lambda$-chains) the sum on the bottom of the above expression evaluates to $t^{\ell(w')}$. We now have 
\begin{align*}
\sum_{(w',T')\in f^{-1}(\sigma)}(-1)^{\ell(w'T')}t^{\frac{1}{2}(\ell(w'T')-\ell(w')-|T'|)+\ell(w')}(t-1)^{|T'|}
\end{align*}
Using the length formula $\ell(w'T')=\ell(w')+|T'|+2N(w',T')$, we notice that $\ell(w'T') \equiv \ell(w')+|T'| \textrm{  mod 2}$. Therefore, $(-1)^{\ell(w'T')}=(-1)^{\ell(w')+|T'|}$. The sum then becomes 
\begin{align*}
(-t)^{\ell(w')}\sum_{(w',T')\in f^{-1}(\sigma)} t^{N(w',T')}(1-t)^{|T'|}
\end{align*}
By Proposition 4.3 we get
\begin{align*}
(-t)^{\ell(w')} \cdot t^{\textrm{inv}(\sigma)-\ell(C)}\cdot(1-t)^{\textrm{des}(\sigma)}=\\
(-t)^{\ell(C)+n-a_{0}}\cdot t^{\textrm{inv}(\sigma)-\ell(C)}\cdot(1-t)^{\textrm{des}(\sigma)}=\\
(-1)^{\ell(C)+n-a_{0}}\cdot t^{n-a_{0}+\textrm{inv}(\sigma)}\cdot(1-t)^{\textrm{des}(\sigma)}.
\end{align*}
Where $a_{0}$ is the unique value absent from the the first column, $C$.
\end{proof1}

\section{Whittaker Strong Compression}
\subsection{Outline of results}
In this chapter we compress our HHL-type formula to Tokuyama's formula. This is accomplished by studying the preimage of a SSYT under the \emph{sort} map, which sorts the columns of a HHL filling in an increasing order. We define an algorithm which uniquely generates this preimage. This allows for a representation of two column HHL configurations graphically as a tree. The algorithm uses a sequence of transpositions and considers subsequences as paths in a tree diagram starting at a central \emph{root}. By fixing the right side of an HHL filling, we generate all possible left sides by using our algorithms recursively. By studying the tree diagrams produced by these procedures, we obtain a bracketing rule that outlines the order that terms should be combined to perform compression. Furthermore, this algorithm does this bracketing in a binary manner, i.e., $t^{a}(1-t)^{b}+t^{a+1}(1-t)^{b-1}=t^{a}(1-t)^{b-1}$. After performing the compression at full scale, we show that the resulting statistics are equivalent to the statistics derived via the bijection between GT-patterns and SSYT.

\subsection{Generation algorithm}
\indent We now turn to our generation algorithm. We generate all HHL fillings by ordering the columns of a SSYT from right to left. Given a fixed right column, and a set of entries to be ordered and placed in the left column, we lay out an algorithm which produces all two column HHL configurations. We then iterate this procedure from right to left, taking each newly constructed left column as the fixed right column and consider the set of entries in the column to the left, stopping when we have run out of columns. 
\newline\indent We achieve this goal by producing a binary search tree. We construct a root configuration and a sequence of transpositions specified below. At each node we have the choice to apply a transposition from this sequence or not. The application of a transposition produces a new two column HHL configuration. 
\newline\indent The inputs of our algorithm are a fixed right column $C'$ of size $c'$ and a set $S=\{a_{1}>a_{2}>\ldots>a_{c}\}$ where $c=c'$ or $c=c'+1$ such that $\overrightarrow{S}\overrightarrow{C'}$ is weakly increasing in rows. We now give the construction of the root configuration of our search tree. 
\begin{definition}
Define $\widehat{C}$ to be the column of entries taken from $S$ dependent on $C'$  as follows: Starting with the largest entry of $S$, $a_{1}$, we check if $a_{1}$ is an entry in $C'$. If it is, place it to the left of $a_{1}$ in $C'$. If it is not, place it to the left of the largest entry in $C'$ with empty box to its left if the columns are the same size, otherwise place the entry at the bottom of the left column. Proceed to the second largest entry in $S$ and repeat. After placing the second largest entry, proceed to the third and so on until all entries of $S$ are placed. 
\end{definition}
The following proposition shows that this root configuration is in fact HHL. 
\begin{prop}
Given $S$ and $C'$ satisfying the above conditions, we have that $\widehat{C}C'$ is HHL. 
\end{prop}
\begin{proof}
First, $\widehat{C}C'$ has no diagonal attacking configurations by definition. Furthermore, $\widehat{C}C'$ is weakly increasing in rows since we have equality in all rows which have a common entry in both columns and a strict increase in rows which have distinct entries between these two columns. To see this, assume $C'$ is strictly increasing and $c=c'$. Consider the largest distinct entry in $C'$, $y$. The entry to its left, $x$ in $\overrightarrow{S}$ is greater than or equal to the largest distinct entry in $\overrightarrow{S}$, $y'$. So $y'$ is placed next to $y$. Go to the next largest distinct entry, and repeat. If $c=c'+1$ then the largest distinct entry in $S$ is placed at the bottom of the column, we can then omit that box and use the above argument.
\end{proof}
\indent We now introduce an important property of entries when considering pairs of columns. This definition is necessary for constructing the sequence of transpositions which we will use to generate our search tree. Furthermore, entries which satisfy this property are relevant to the compression we perform. 
\begin{definition}
Call row $i$ a \emph{pivot} if $\widehat{C}(i)<C'(i)$ if $C'(i)$ exists. If $i=c'+1$ then we consider this a pivot row as well. Let $p$ be the number of pivots in columns  $\widehat{C}C'$. If $i$ is a pivot row, call $\widehat{C}(i)$ and $C'(i)$ \emph{pivot entries}.
\end{definition} 
 Let $b_{1}>b_{2}>\ldots >b_{p}$ be the pivot entries in the left column. The pivot entries strictly less that $a_{i}$ are denoted $b_{j_i}>\ldots > b_{p}$ with $j_{1} \leq j_{2} \leq \ldots\leq j_{c}$. We have the following sequence of transpositions:
\[
\beta=
\begin{matrix}
(b_{j_1}<a_{1}) & (b_{j_{1}+1}<a_{1}) & (b_{j_1+2}<a_{1}) &\ldots & (b_{p}<a_{1})\\
(b_{j_2}<a_{2}) & (b_{j_{2}+1}<a_{2}) & (b_{j_2+2}<a_{2}) &\ldots & (b_{p}<a_{2})\\
\ldots \\
(b_{j_c}<a_{c}) & (b_{j_{c}+1}<a_{c}) & (b_{j_c+2}<a_{c}) &\ldots & (b_{p}<a_{c})\\
\end{matrix}
\]
and, 
\[
\beta_{(b_{j_i+k}<a_{l})}=
\begin{matrix}
(b_{j_{i}+k+1}<a_{l}) & (b_{j_{i}+k+2}<a_{l}) & (b_{j_{i}+k+3}<a_{l}) &\ldots & (b_{p}<a_{l})\\
(b_{j_{i+1}}<a_{l+1}) & (b_{j_{i+1}+1}<a_{l+1}) & (b_{j_{i+1}+2}<a_{l+1}) &\ldots & (b_{p}<a_{l+1})\\
\ldots \\
(b_{j_c}<a_{c}) & (b_{j_{c}+1}<a_{c}) & (b_{j_c+2}<a_{c}) &\ldots & (b_{p}<a_{c})\\
\end{matrix}
\]
for a transposition $(b_{j_i+k}<a_{l}) \in \beta$. Consider the total order $\preceq$ given by reading the sequence $\beta$ from left to right and top to bottom. 
\begin{definition}
Let $\mathcal{A}(\beta)$ (or $\mathcal{A}(\beta_{(b_{j_i+k}<a_{l})}))$ be the set of all subsequences such that there is only one transposition from each row of $\beta$ (or $\beta_{(b_{j_i+k}<a_{i})}$). 
\end{definition}
Here we recall that the right column is fixed, so a transposition $(a<b) \in \beta$ acts on the left column by switching the positions of $a$ and $b$ in the left column. 
\begin{definition}
Given $CC'$ a two column HHL configuration, and a sequence of transpositions $\sigma_{1}\ldots \sigma_{k}\in \mathcal{A}(\beta)$ define an action on $CC'$ given by $\sigma_{1}\ldots \sigma_{k}.CC'=\sigma_{k}\ldots \sigma_{1}CC'$ if $\sigma_{k}\ldots \sigma_{1}CC'$ is a two column HHL configuration, and $\sigma_{1}\ldots \sigma_{k}.CC'=\emptyset$ if $\sigma_{k}\ldots \sigma_{1}CC'$ is not  a two column HHL configuration. 
\end{definition}
\begin{definition}
A transposition of entries, $\sigma_{k}$, is called \emph{legal} with respect to a two column HHL configuration $CC'$, if we have $\sigma_{k}CC'$ is HHL.  
\end{definition}
\indent We now introduce a condition on two column HHL configurations known as the non-overlapping condition. This condition, as we will see in the next subsection, detects when our sum completely cancels. We wish to study the intervals formed between the pivot entries in our two column HHL configuration. If the intersection between these intervals is empty for all pairs of intervals then we do not have a legal switch between pivot entries in $\beta$. If it is not empty, then there exists a legal switch between pivot entries, and so the $t$-coefficients associated to this HHL fillings will cancel when paired. 
\newline\indent Let $q_{1},\ldots, q_{p}$ be the pivot rows of $\widehat{C}C'$ and $b_{j}$ be the pivot entry corresponding to pivot $q_{j}$. Let $d_{j}=C'(q_{j})$ if it exists (otherwise take $d_{j}=\infty$). Define,
\[
P_{j,k}=[b_{j},d_{j}] \cap [b_{k},d_{k}]
\]
\begin{definition}If $p>1$ then we say $\widehat{C}C'$ satisfies the \emph{Non-overlapping Condition} if $P_{j,k}=\emptyset$ for all pairs $j \neq k$ in $[p]$. If $p=0$ or $1$ then the \emph{Non-overlapping Condition} is automatically satisfied. 
\end{definition}
We now present the main proposition of this subsection. Below, we show that the generation algorithm uniquely generates all two column HHL configurations. Once we have this result, we can iterate the above algorithm by considering columns from right to left, extending the result to the entire preimage of a SSYT. 
\begin{prop}
$\mathcal{A}(\beta).\widehat{C}$ uniquely generates all permutations of $\widehat{C}$ giving an HHL two column configuration when placed to the left of $C'$. 
\end{prop}
\begin{proof}
Let $C$ be a left column such that $CC'$ is an HHL two column configuration. Let $\sigma$ be the unique sequence of transpositions applied to $\widehat{C}$ which places entries in their corresponding position in $C$ from largest to smallest. We need to show that $\sigma$ is in $\mathcal{A}(\beta)$. In other words, it suffices to show that each transposition in this sequence involves a pivot entry and that each transposition is legal. 
\newline\indent We first show that each transposition involves a pivot entry. To do this, we perform induction on the number of transpositions in $\sigma$. First, consider when $\sigma$ is a single transposition. Let $x$ be the largest misaligned entry between $\widehat{C}$ and $C$. We have that $x$ must be switched with a pivot entry since otherwise we have a strict decrease in the row $x$ occupies in $C$. Therefore $\sigma$ switches a pivot entry with an arbitrary entry. Now assume that if $\sigma=\sigma_{i-1}\ldots\sigma_{1}$ then each $\sigma_{k}$ is a transposition which involves a pivot entry for all $1\leq k\leq i-1$. Consider when $\sigma=\sigma_{i}\ldots\sigma_{1}$. Let $x$ be the largest misaligned entry between $C$ and $\sigma_{i-1}\ldots\sigma_{1}\widehat{C}$. Let $j$ be the row such that $C(j)=x$. Let $a=\sigma_{i-1}\ldots\sigma_{1}\widehat{C}(j)$ and let $b=C'(j)$. Assume for contradiction that $a$ is a non-pivot entry. By induction $\sigma_{i-1},\ldots,\sigma_{1}$ are transpositions between entries larger than $x$ and pivot entries, and since $a<x$ then $a$ has not been moved. Therefore, by the construction of $\widehat{C}$, $a=b$. This implies that $CC'$ is not weakly increasing in rows since $C(j)=x$ and $C'(j)=a$. Therefore, $a$ is a pivot. 
\newline\indent Furthermore all transpositions applied in such a manner are legal. Assume some transposition $(a<b)$ of entries is applied to the left column of an intermediate HHL filling. We need to check that the resulting filling is HHL. Since $b$ is placed in its final position with respect to $C$, and since $CC'$ is HHL, the row in which $b$ is placed is also weakly increasing. Likewise, if $x$ is the entry to the right of $b$ in $C'$ then $a$ is placed to the left of $x$; but since $a<b$ and $b\leq x$ we have $a<x$. In addition, we do not introduce any attacking configurations. Indeed, it suffices to check this for the entry $b$; but this entry is placed in its final position with respect to $C$ which must be non-attacking. 
\end{proof}
\indent We now turn to some combinatorics related to $\mathcal{A}(\beta)$ and the Non-overlapping Condition. By the previous proposition, for a two column HHL configuration $CC'$, let $\sigma\in\mathcal{A}(\beta)$ such that $C=\sigma.\widehat{C}$. We have the following properties which lead to a property we refer to as \emph{Cycle Structure}.
\begin{lemma}Let $C$ be a column such that $CC'$ is HHL. If there exists a legal transposition between pivot entries in $C$ then the Non-overlapping Condition is not satisfied. 
\end{lemma}
\begin{proof} By Proposition 4.2, suppose $\sigma.\widehat{C}C'=CC'$ for $\sigma \in \mathcal{A}(\beta)$. We perform induction on the length of $\sigma$. If $\sigma=\emptyset$ and there exists a legal transpostion between pivot entries $(x<y)$ in $\widehat{C}C'$ then the Non-overlapping Condition is not satisfied. Indeed, if $x'$ and $y'$ are the entries to the right of $x$ and $y$ and $(x<y)$ is legal then $y<x'$. By the construction of $\widehat{C}$, $x'<y'$ so we have $x<y<x'<y'$ which implies $[x,x']\cap[y,y']\neq \emptyset$. 
\newline\indent Assume for $\sigma=\sigma_{1}\ldots\sigma_{k-1}$ with $\sigma.\widehat{C}C'=CC'$ such that $CC'$ is HHL, that if there exists a legal transposition between pivot entries then the Non-overlapping Condition is not satisfied. Let $\sigma=\sigma_{1}\ldots\sigma_{k}$ such that $\sigma.\widehat{C}C'=CC'$ is HHL and assume there exists a legal transposition between pivot entries, $(x<y)$. Apply $\sigma_{k}$ to $CC'$. If $\sigma_{k}$ does not involve $x$ or $y$ then we can use the induction hypothesis since $x$ and $y$ were not moved. If $\sigma_{k}$ involves both $x$ and $y$ then we can use the induction hypothesis since $(x<y)$ remains a legal transposition between pivot entries with respect to $\sigma_{k}.CC'$. If $\sigma_{k}$ involves $x$, then let $a\neq y$ be the entry $x$ is switched with. If $a$ is a pivot entry then either $a<x$ or $x<a$. In both cases, this is a legal transposition between pivot entries with respect to $\sigma_{k}.CC'$ so we can use the induction hypothesis. Finally, if $a>x$ and $a$ is a non-pivot entry, then $y$ is less than the entry to the right of $x$ in $\sigma_{k}.CC'$ since $(x<y)$ is legal in $CC'$. Then $(x<y)$ is a legal transposition between pivot entries with respect to $\sigma_{k}.CC'$, which implies we can use the induction hypothesis. The same arguments work analogously for if $\sigma_{k}$ involves just $y$.
\end{proof}
\begin{lemma}If the Non-overlapping Condition is satisfied then there exists a unique pivot in each disjoint cycle of $\sigma\in \mathcal{A}(\beta)$ as a permutation. 
\end{lemma}
\begin{proof}
If the Non-overlapping Condition is satisfied then pivot entries must be switched with non-pivot entries by Lemma 4.1. This implies that each disjoint cycle of $\sigma$ has exactly one pivot.
\end{proof}
\begin{prop}
If the Non-overlapping Condition is satisfied then $(b_{j_{i}}<a_{i})$ is the only possible legal transposition involving $a_{i}$ which can be applied to $CC'$.
\end{prop}
\begin{proof}
Suppose there exists another pivot entry $b$ less than $b_{j_{i}}$ such that $(b<a_{i})$ is legal with respect to $CC'$. This implies the entry to the right of $b$ is larger than $a_{i}$. In addition, the entry to the right of $b_{j_{i}}$ is larger than $a_{i}$. This implies there exists a legal transposition between pivot entries in $CC'$ which by Lemma 4.1 implies the Non-overlapping Condition is not satisfied.
\end{proof}

We now turn our attention to the generation tree for a two column configuration. We then extend it an entire SSYT.
\newline\textbf{Generation Tree Construction: }
\begin{enumerate}
\item Begin with $\widehat{C}C'$ as the root of our tree. Start with the smallest transposition in $\beta$ with respect to $\preceq$.
\item For a transposition $(p<q)$ and for all nodes with no children $CC'$, do the following: If $(p<q)$ is legal then we can apply it to $CC'$ or not. If we apply it, add a node corresponding to the filling $(p<q).CC'$, and connect this node to the parent with an edge labeled $(p<q)$. If we do not apply $(p<q)$ then add a node corresponding to $CC'$ and connect this node to the root with an edge labeled $\overline{(p<q)}$. If $(p<q)$ is not legal then add a node corresponding to $CC'$ and connect this node to the parent with an edge labeled $\overline{(p<q)}$. 
\item Go to the next largest transposition in $\beta$ and repeat (2) until we have considered all transpositions in $\beta$.
\end{enumerate}
\indent Starting at the right most column we can iterate this procedure from right to left appending columns on the left. Once we have gone through all of the columns, we assign $t$-coefficients to the leaves of this generation tree for which our statistics on HHL fillings are defined. The branching of the generation tree is determined by whether or not we applied a transposition, so fillings with the same parent node are paired and their $t$-coefficients are added. This sum is then assigned as the $t$-coefficient of the parent node. For examples of this construction, refer to the Appendix. 
\subsection{Proof of compression}
In this section we perform the proof of compression. Once the generation tree is constructed and $t$-coefficients are assigned to the leaves of this tree, we perform a breadth-first tree traversal, summing the $t$-coefficients associated to each child of a given node, and assign that sum to the parent node. 
\newline\indent We first give a number of useful lemmas regarding the behavior of the inversion and descent statistics after transpositions from $\beta$ are applied to the corresponding column. We pair fillings that differ by a single transposition in the left column. The first step is to compress fillings which have larger entries in identical position. This amounts to fixing an initial subsequence $\sigma\in\mathcal{A}(\beta)$ and performing the compression for all HHL fillings $CC'=\sigma\gamma.\widehat{C}C'$ where $\sigma\gamma \in \mathcal{A}(\beta)$. Once we have this partial result, we extend it to an entire column, and further to an arbitrary number of columns. Finally, we are able to express the sum of all HHL fillings which sort to the same SSYT in terms of a single canonical HHL filling which lies in the preimage of the SSYT. 
\newline\indent Refer to Definition 3.2 for the definitions of the inversion and descent statistic. We have the following flowchart indicating the sequence of implications that lead to the main theorem. 
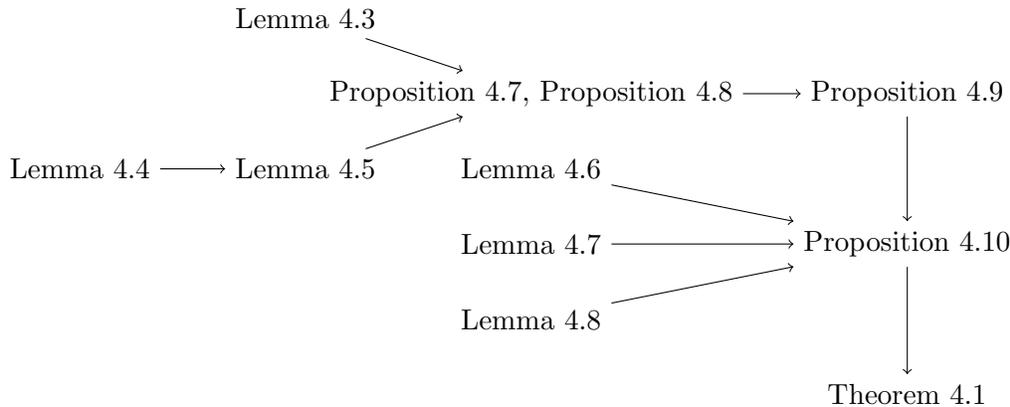
\begin{tikzpicture}[node distance=0cm]
\node(1)[yshift=1cm,xshift=1cm]{\text{Lemma 4.3}};
\node(2)[xshift=4cm]{\text{Proposition 4.7, Proposition 4.8}};
\node(3)[yshift=-1cm,xshift=-2cm]{\text{Lemma 4.4}};
\node(4)[yshift=-1cm,xshift=1cm]{\text{Lemma 4.5}};
\node(5)[xshift=9cm]{\text{Proposition 4.9}};
\node(6)[yshift=-1cm,xshift=4cm]{\text{Lemma 4.6}};
\node(7)[yshift=-2cm,xshift=4cm]{\text{Lemma 4.7}};
\node(8)[yshift=-3cm,xshift=4cm]{\text{Lemma 4.8}};
\node(9)[yshift=-2cm,xshift=9cm]{\text{Proposition 4.10}};
\node(10)[yshift=-4cm,xshift=9cm]{\text{Theorem 4.1}};
\draw[->] (1) edge (2);
\draw[->] (3) edge (4);
\draw[->] (4) edge (2);
\draw[->] (2) edge (5);
\draw[->] (6) edge (9);
\draw[->] (7) edge (9);
\draw[->] (8) edge (9);
\draw[->] (5) edge (9);
\draw[->] (9) edge (10);
\end{tikzpicture}
\newline\indent The following notation applies to the rest of this subsection. Let $T$ be an HHL filling. We read the columns of $T$ from left to right. Define $T[a,b]$ to be the section of $T$ comprised of columns $a$ through $b$ (inclusive) of $T$. Denote the left-most column by $C$ and denote the column to its right by $C'$, where each column is viewed as both a set and a sequence. Recall that $\widehat{C}$ is a permutation of $C$ whose construction is dictated by $C'$ as in Definition 4.1. Let $T'$ be the HHL filling composed of the columns strictly to the right of $C'$. 
\newline\indent Given a generation tree with root $\widehat{C}C'$, let $CC'=\sigma.\widehat{C}C'$ which is uniquely determined by Proposition 4.2. Let $(p_{*}<q_{*})$ be the transposition labeling the edge immediately above the highest $CC'$ in the generation tree, i.e. $(p_{*}<q_{*})=\sigma_{k}$ for $\sigma=\sigma_{1}\ldots\sigma_{k}$.  Let $r(a)$ denote the row of $a$ in $C$. Define $C'_{>r(a)}$ to be the subsequence of $C'$ consisting of all entries below $r(a)$. Analogously define $C'_{<r(a)}$. 

\begin{lemma}
Given two columns $CC'$, and a legal transposition $(p<q) \succ (p_{*}<q_{*})$ where $q$ is a non-pivot entry, we have $\emph{des}((p<q)CC')=\emph{des}(CC')+1$.
\end{lemma}
\begin{proof}
Since $(p<q)$ is after the last transposition used to generate $C$ and $q$ is a non-pivot entry then $q$ has not been moved. Therefore by switching the pivot entry $p$ with the non-pivot entry $q$, we add a descent in row $r(q)$.
\end{proof}
Define,
\[
N_{a,b}[C'_{>r(a)}\cap C]=\#\{ k \text{ }|\text{ } k\in C'_{>r(a)}\cap C, \text{ } a<k<b\}
\]
\begin{lemma}
Given two columns $CC'$ and a legal transposition $(p<q) \succeq (p_{*}<q_{*})$, such that $r(p)<r(q)$. We have $\emph{inv}((p<q)CC')=\emph{inv}(CC')+N_{p,q}[C'_{>r(p)}\cap C'_{<r(q)}]$
\end{lemma}
\begin{proof}
If there exists an entry, $s$ in $C'$, which participates in an inversion with respect to row $r(q)$ in $CC'$, then it remains an inversion with respect to row $r(p)$ in $(p<q)CC'$ since $p<q$. If $s$ participates in an inversion with respect to row $r(p)$ in $CC'$ then it remains an inversion with respect to row $r(p)$ or $r(q)$ in $(p<q)CC'$ depending on the value of $s$. If $s<q$ then the inversion is with respect to $r(p)$ in $(p<q)CC'$. If $s>q$ then the inversion is with respect to $r(q)$ in $(p<q)CC'$. Finally, if $p<s<q$ and lies below $r(p)$, but above $r(q)$ in $CC'$ then we gain an inversion with respect to row $r(p)$ in $(p<q)CC'$. The number of such entries which satisfy this criterion is $N_{p,q}[C'_{>r(p)}\cap C'_{<r(q)}]$. 
\end{proof}
Furthermore, define
\[
m(CC',q)=\sum_{\substack{a\in C\setminus C'}}N_{a,\min(C'(r(a)),q)}[C'_{>r(a)}\cap C].
\] 
\begin{lemma}
Given two columns $CC'$, a legal transposition $(p<q)\succeq (p_{*}<q_{*})$ as above such that $r(p)<r(q)$, and that there are no legal switches between pivots after $(p<q)$ in $\beta$, then 
\[
\emph{inv} ( (p<q)CC'T')+m((p<q)CC',q)=\emph{inv} ( CC'T')+m(CC',q).
\]
\end{lemma}
\begin{proof}
The left hand side of the equation is,
\[
\text{inv} ( (p<q)CC'T')+m((p<q)CC',q).
\]
By direct substitution for $m((p<q)CC',q')$ and using Lemma 4.4 gives, 
\[
\text{inv}(CC'T')+N_{p,q}[C'_{>r(q)}\cap C'_{<r(p)}]+\sum_{\substack{a\in C\setminus C'}}N_{a,\min(C'(r(a)),q)}[C'_{>r(a)}\cap (p<q)C].
\]
We can consider the term in the sum where $a=p$ separately, which gives
\[
\text{inv}(CC'T')+N_{p,q}[C'_{>r(q)}\cap C'_{<r(p)}]+N_{p,\min(C'(r(p)),q)}[C'_{>r(p)}\cap (p<q)C]+
\]
\[
\sum_{\substack{a\in C\setminus C'\\a\neq p}}N_{a,\min(C'(r(a)),q)}[C'_{>r(a)}\cap (p<q)C].
\]
\indent It is important to note that the row indices are with respect to $(p<q)C$. We next convert to row indices that are with respect to $C$. Notice that switching $p$ and $q$ does not affect the terms in the sum corresponding to $a\neq p$ because 
\[
\{ k \text{ }|\text{ } k\in C'_{>r(a)}\cap C, \text{ } a<k<\min(C'(r(a)),q)\}=
\]
\[
\{ k \text{ }|\text{ } k\in C'_{>r(a)}\cap (p<q)C, \text{ } a<k<\min(C'(r(a)),q)\}.
\]
So we can substitute $C$ for $(p<q)C$ in the summation. In addition, 
$$N_{p,\min(C'(r(p)),q)}[C'_{>r(p)}\cap (p<q)C]=N_{p,q}[C'_{>r(q)}\cap C],$$ 
which follows from the facts that $\min(C'(r(p)),q)=q$ since $(p<q)$ is a legal transposition and $r(p)$ with respect to $(p<q)C$ equals $r(q)$ with respect to $C$. So the above sum can be reinterpreted with respect to $C$ as,
\[
\text{inv}(CC'T')+N_{p,q}[C'_{>r(p)}\cap C'_{<r(q)}]+N_{p,q}[C'_{>r(q)}\cap C]+
\]
\[
\sum_{\substack{a\in C\setminus C'\\a\neq p}}N_{a,\min(C'(r(a)),q)}[C'_{>r(a)}\cap C].
\]
\indent Furthermore, $N_{p,q}[C'_{>r(p)}\cap C'_{<r(q)}]=N_{p,q}[C'_{>r(p)}\cap C'_{<r(q)}\cap C]$ since if there exists a $y \in C'\setminus C$ between rows $r(p)$ and $r(q)$ with respect to $C$ such that $p<y<q$ then there exists a legal switch between pivots, $(y'<p)$, where $y'$ is the entry in $C$ paired with $y$ in $\widehat{C}$. To see this, we first show that $y'$ has not been moved. Suppose an entry $x\neq y'$ is to the left of $y$. We have $x>q>y$, where the first inequality follows from $x$ being moved before $q$, and the second inequality comes from the assumption on $y$. This is a contradiction since $x$ is to the left $y$ and $CC'$ is HHL. Therefore $y'$ is in the same row as $y$ which implies $(y'<p)$ is a legal transposition of pivots after $(p<q)$. So there exists no such $y$, and we can conclude the identity. 
\newline\indent After substituting $N_{p,q}[C'_{>r(p)}\cap C'_{<r(q)}\cap C]$ for $N_{p,q}[C'_{>r(p)}\cap C'_{<r(q)}]$, we get,
\[
\text{inv}(CC'T')+N_{p,q}[C'_{>r(p)}\cap C'_{<r(q)}\cap C]+N_{p,q}[C'_{>r(q)}\cap C]+
\]
\[
\sum_{\substack{a\in C\setminus C'\\a\neq p'}}N_{a,\min(C'(r(a)),q)}[C'_{>r(a)}\cap C].
\]
Combine $N_{p,q}[C'_{>r(p)}\cap C'_{<r(q)}\cap C]+N_{p,q}[C'_{>r(q)}\cap C]$ to get, 
\[
\text{inv}(CC'T')+N_{p,q}[C'_{>r(p)}\cap C]+
\sum_{\substack{a\in C\setminus C'\\a\neq p}}N_{a,\min(C'(r(a)),q)}[C'_{>r(a)}\cap C].
\]
Bring $N_{p,q}[C'_{>r(p)}\cap C]$ back into the sum to get,
\[
\text{inv}(CC'T')+
\sum_{\substack{a\in C\setminus C'}}N_{a,\min(C'(r(a)),q)}[C'_{>r(a)}\cap C].
\]
By direct substitution, the above sum gives the Lemma.
\end{proof}
We now define the sum of the terms beneath a node in the generation tree. This is the sum of all $t$-coefficients which have entries less than $q$ in differing positions in the left column with respect to $CC'$.
\begin{definition} Define the sum at a node to be,
\[
S(CC'T',(p<q))=\sum_{\substack{\gamma \in \mathcal{A}(\beta_{(p<q)})\\ \gamma.CC'T'\text{ }\\ HHL}}(-1)^{\ell(\gamma.C)}t^{\emph{inv} (\gamma.CC'T')}(1-t)^{\emph{des}(\gamma.CC' T')}.
\]
\end{definition}
We have the following proposition. 
\begin{prop} Given two columns $CC'$ HHL, such that $\widehat{C}C'$ satisfies the Non-overlapping condition, and a transposition $(p<q)\succeq(p_{*}<q_{*})$ we have 
\[
S(CC'T',(p<q))=(-1)^{\ell(C)}t^{\emph{inv}(CC'T')+m(CC',q)}(1-t)^{\emph{des}(CC'T')}
\]
\end{prop}
\textbf{Remark: } Note that the left hand side depends on $p$ while the right hand side does not. This follows from the fact that if the Non-overlapping Condition is satisfied then each entry in the left column lies in at most one pivot interval. So each row of $\beta_{(p<q)}$ has at most one transposition that can be applied to $C$. This eliminates the dependency on $p$. 
\begin{proof}
Suppose $(p<q)$ is the last transposition in $\beta$, so $\beta_{(p<q)}$ is empty. This implies 
\[
\sum_{\substack{\gamma \in \mathcal{A}(\beta_{(p<q)})\\ \gamma.CC'T'\text{ }\\ HHL}}(-1)^{\ell(\gamma.C)}t^{\text{inv} (\gamma.CC'T')}(1-t)^{\text{des}(\gamma.CC' T')}=(-1)^{\ell(C)}t^{\text{inv} ( CC'T')}(1-t)^{\text{des}(CC' T')},
\]
since there is only one term in this sum corresponding to $\gamma=\emptyset$. By the definition of $m(CC',q)$ we have,
\[
m(CC',q)=\sum_{\substack{a\in C\setminus C'}}N_{a,\min(C'(r(a)),q)}[C'_{>r(a)}\cap C].
\]
Note that $(p<q)$ is the transposition between the two smallest values of $C$ so we have
\[
N_{a,\min(C'(r(a)),q)}[C'_{>r(a)}\cap C]=0,
\]
for all $a \in C\setminus C'$.
\newline\indent For the induction step, let $(p'<q')$ be a transposition from the row directly below row $q$ in $\beta_{(p<q)}$. In other words, $q'<q$ and there does not exist an $x$ in $C$ such that $q'<x<q$. We can assume without loss of generality that $(p'<q')$ is a legal transposition. To see this, assume $(p'<q')$ is not legal. We have $S(CC'T',(p'<q'))=S(CC'T',(p<q))$ since $\{\gamma.CC'T':\text{ } \gamma\in\mathcal{A}(\beta_{(p'<q')})\}=\{\gamma.CC'T':\text{ } \gamma\in\mathcal{A}(\beta_{(p<q)})\}$. In addition, we have $m(CC',q')=m(CC',q)$ because $N_{a,\min(C'(r(a)),q')}[C'_{>r(a)}\cap C]=N_{a,\min(C'(r(a)),q)}[C'_{>r(a)}\cap C]$ for all $a\in C\setminus C'$. This follows from the fact that if $(p'<q')$ is not legal then either $q'$ is greater than the entry to the right of $p'$ or $q'$ is above $p'$. 
\newline\indent We have, 
\[
S(CC'T',(p<q))=S(CC'T',(p'<q'))+S((p'<q')CC'T',(p'<q')).
\]
This is the sum of the two nodes directly below $CC'T'$, and corresponds to whether we applied the transposition $(p'<q')$ to $CC'T'$. 
By induction, 
\[
S(CC'T',(p'<q'))+S((p'<q')CC'T',(p'<q'))=
\]
\[
(-1)^{\ell(C)}t^{\text{inv} ( CC'T')+m(CC',q')}(1-t)^{\text{des}(CC' T')}+
\]
\[
(-1)^{\ell((p'<q')C)}t^{\text{inv} ( (p'<q')CC'T')+m((p'<q')CC',q')}(1-t)^{\text{des}((p'<q')CC' T')}=
\]
\[
(-1)^{\ell(C)}t^{\text{inv} ( CC'T')+m(CC',q')}(1-t)^{\text{des}(CC' T')}-
\]
\[
(-1)^{\ell(C)}t^{\text{inv} ( (p'<q')CC'T')+m((p'<q')CC',q')}(1-t)^{\text{des}(CC' T')+1}.
\]
Since $(p'<q')$ is a legal transposition then $q'$ is a non-pivot entry and $r(p')<r(q')$ so Lemma 4.3 and Lemma 4.5 gives, 
\[
(-1)^{\ell(C)}t^{\text{inv} ( CC'T')+m(CC',q')}(1-t)^{\text{des}(CC' T')}-
\]
\[
(-1)^{\ell(C)}t^{\text{inv} ( (p'<q')CC'T')+m((p'<q')CC',q')}(1-t)^{\text{des}(CC' T')+1}=
\]
\[
(-1)^{\ell(C)}t^{\text{inv} ( CC'T')+m(CC',q')}(1-t)^{\text{des}(CC' T')}(1-(1-t))=
\]
\[
(-1)^{\ell(C)}t^{\text{inv} ( CC'T')+m(CC',q')+1}(1-t)^{\text{des}(CC' T')}.
\]
Finally, since $q'<q$, we have $m(CC',q')+1=m(CC',q)$. This follows from the Non-overlapping Condition being satisfied. Since $q'$ is in only one pivot interval, namely the one involving $p'$, for all $a\neq p'$, we have that $N_{a,\min(C'(r(a)),q')}[C'_{>r(a)}\cap C]=N_{a,\min(C'(r(a)),q)}[C'_{>r(a)}\cap C]$. For $a=p'$ we have, $N_{p',q'}[C'_{>r(p')}\cap C]+1=N_{p',q}[C'_{>r(p')}\cap C]$ because there are no entries between $q$ and $q'$.
\end{proof}
Here we perform the compression from a node down in our tree given that the Non-overlapping Condition is not satisfied. 
\begin{prop}
Given two columns which do not satisfy the Non-overlapping Condition and a transposition $(p<q)\succeq (p_{*}<q_{*})$ such that there exists a legal transposition between pivots $(p'<q')$ in $\beta_{(p<q)}$, we have 
\[
S(CC'T',(p<q))=0.
\]
\end{prop}
\begin{proof}
Let $(p'<q')$ be the last legal transposition between pivot entries in $\beta_{(p<q)}$. By previous result,
\[
S(CC'T',(p'<q'))=
(-1)^{\ell(C)}t^{\text{inv}(CC'T')+m(CC',q')}(1-t)^{\text{des}(CC'T')}
\]
and 
\[
S((p'<q')CC'T',(p'<q'))=
\]
\[
(-1)^{\ell((p'<q')C)}t^{\text{inv}((p'<q')CC'T')+m((p'<q')CC',q')}(1-t)^{\text{des}((p'<q')CC'T')}.
\]
If $r(p')<r(q')$ with respect to $CC'$ then by Lemma 4.5, we have the following:
\[
\text{inv}((p'<q')CC'T')+m((p'<q')CC',q')=\text{inv}(CC'T')+m(CC',q').
\]
On the other hand, if $r(p')>r(q')$ with respect to $CC'$, we can still use Lemma 4.5 but we have to set $\widetilde{C}C'=(p'<q')CC'$ and consider $(p'<q')\widetilde{C}C'$.
\newline\indent Since $(p'<q')$ is a transposition of pivot entries we have $ \text{des}((p'<q')CC'T')=\text{des}(CC'T')$, and by definition of length we have $\ell((p'<q')C)\equiv\ell(C)+1\text{ mod }2$. This gives
\[
S(CC'T',(p'<q'))+S((p'<q')CC'T',(p'<q'))=0.
\]
So all nodes below $(p<q)$ have $t$-coefficient equal to $0$, therefore $S(CC'T',(p<q))=0$.
\end{proof}
\indent  Recall that $b_{j}$ and $d_{j}$ are the end points of the pivot intervals with respect to the two left-most columns and $p$ is the number of pivots rows in these two columns. Furthermore, recall that $\widehat{C}$ is the left-most column corresponding to matching all common entries between it and its neighboring column to the right and then taking all distinct entries and matching them according to the order given by the distinct entries in the neighboring column. Define  
\[
\check{N}(i)=\sum_{j=1}^{p}N_{b_{j},d_{j}}[C_{i_{>r(b_{j})}}\cap \widehat{C}]
\]
where $b_{j} \in \widehat{C}\setminus C_{i}$ and $d_{j}=C_{i}(r(b_{j}))$ if it exists, otherwise take $d_{j}=\infty$. The row indices are with respect to $\widehat{C}$. 
We have the following,
\begin{prop} Fix a column $C_{i}$ of a HHL filling $T$. Let $T'=T[i+1,(\lambda+\rho)_{1}]$. We have, 
\[
\sum_{\substack{\gamma \in \mathcal{A}(\beta)\\ \gamma.\widehat{C}C_{i}T'\\ HHL}} (-1)^{\ell(\gamma.\widehat{C})}t^{\emph{inv}(\gamma.\widehat{C}C_{i}T')}(1-t)^{\emph{des}(\gamma.\widehat{C}C_{i}T')}
\]
\[
=\begin{cases}
(-1)^{\ell(\widehat{C})}t^{\emph{inv}(\widehat{C}C_{i}T')+\check{N}(i)}(1-t)^{\emph{des}(\widehat{C}C_{i}T')} & \text{if the Non-overlapping Condition is satisfied}\\
0 & \text{otherwise}
\end{cases}
\]
\end{prop}
\begin{proof}
Set $C=\widehat{C}$, $C'=C_{i}$ in Proposition 4.4 and Proposition 4.5, and note that $m(\widehat{C}C_{i},q)=\check{N}(i)$ when $q>\max(\widehat{C})$.
\end{proof}
With the above conventions, define 
\[
\hat{N}(i)=\sum_{j=1}^{p}N_{b_{j},d_{j}}[C_{i_{<r(b_{j})}}\cap \widehat{C}]
\]
\begin{lemma}
If $\widehat{C}C_{i}$ satisfies the Non-overlapping Condition, then $\emph{inv}(\widehat{C}C_{i}T')=\emph{inv}(C_{i}T')+\hat{N}(i)$.
\end{lemma}
\begin{proof}
By the definition of $\widehat{C}$ we have that all common entries between $\widehat{C}$ and $C_{i}$ are matched, while the distinct entries are paired according to their relative values. The rows with distinct entries contribute to the inversion statistic. This contribution is $\hat{N}(i)$.
\end{proof}
For column $i$ define $p(i)$ to be the number of descents in $\widehat{C}C_{i}$. We have 
\begin{lemma}
$\emph{des}(\widehat{C}C_{i}T')=\emph{des}(C_{i}T')+p(i)$.
\end{lemma}
\begin{proof}
The contribution of the descents between the two left most columns is $p(i)$. 
\end{proof}
\begin{lemma}
Fix a column $C_{i}$ and consider a two column configuration $\widehat{C}C_{i}$. We have 
\[
\ell(\widehat{C})=\ell(C_{i})+\hat{N}(i)-\check{N}(i).
\]
\end{lemma}
\begin{proof}
By replacing the entries in the pivot rows of $\widehat{C}$ with the entries to their right we increase the length of the column by $\check{N}(i)$ and decrease the length of the column by $\hat{N}(i)$. 
\end{proof}
\indent Here we perform the compression over several columns of an HHL filling. It relies on an inductive argument, which follows from the recursive nature of our generation algorithm. For a fixed right side of an HHL filling we can generate all possible left sides by producing columns which are HHL according to the our generation algorithm. For each column produced we can then repeat our procedure from right to left, generating all HHL fillings. We then compress using the result for single columns in order from left to right.  
 
\begin{prop} 
Let $T'=T[l+2,(\lambda+\rho)_{1}]$ for $l \geq 1$. Fix column $C_{l+1}$, such that $C_{l+1}T'$ is HHL. Assume the $l$ left most columns of $T$ satisfy the Non-overlapping Condition. We have,
\[
\sum_{\substack{C_{1}, \ldots, C_{l}: \\ C_{1}\ldots C_{l}C_{l+1}T'\text{ } HHL}}(-1)^{\ell(C_{1})}t^{\emph{inv}(C_{1}\ldots C_{l+1}T')}(1-t)^{\emph{des}(C_{1}\ldots C_{l+1}T')}=
\]
\[
(-t)^{\sum_{i=2}^{l}(\hat{N}(i)+\check{N}(i))}(1-t)^{\sum_{i=2}^{l}p(i)}(-1)^{\ell({C}_{l+1})}t^{\emph{inv}(C_{l+1}T')}(1-t)^{\emph{des}(C_{l+1}T')}
\]
If the Non-overlapping Condition is not satisfied for the $l$ left most columns then the sum is $0$
\end{prop}
\begin{proof}
For when $l=1$ we first use Proposition 4.6 to get 
\[
\sum_{\substack{C_{1}: \\ C_{1}C_{2}T'\text{ } HHL}}(-1)^{\ell(C_{1})}t^{\emph{inv}(C_{1}C_{2}T')}(1-t)^{\emph{des}(C_{1}C_{2}T')}=
\]
\[
(-1)^{\ell(\widehat{C})}t^{\emph{inv}(\widehat{C}C_{i}T')+\check{N}(i)}(1-t)^{\emph{des}(\widehat{C}C_{i}T')}.
\]
By Lemma 4.6, Lemma 4.7, and Lemma 4.8 we get 
\[
(-1)^{\ell(\widehat{C})}t^{\emph{inv}(\widehat{C}C_{i}T')+\check{N}(i)}(1-t)^{\emph{des}(\widehat{C}C_{i}T')}=
\]
\[
(-1)^{\ell(C_{2})+\hat{N}(2)-\check{N}(2)}t^{\text{inv}(C_{i}T')+\hat{N}(2)+\check{N}(2)}(1-t)^{\text{des}(C_{i}T')+p(2)}=
\]
\[
(-1)^{\ell(C_{2})+\hat{N}(2)+\check{N}(2)}t^{\text{inv}(C_{i}T')+\hat{N}(2)+\check{N}(2)}(1-t)^{\text{des}(C_{i}T')+p(2)}=
\]
\[
(-t)^{\hat{N}(2)+\check{N}(2)}(1-t)^{p(2)}(-1)^{\ell(C_{2})}t^{\text{inv}(C_{2}T')}(1-t)^{\text{des}(C_{2}T')}
\]
\indent For the induction step, we have by hypothesis
\[
\sum_{\substack{C_{1}, \ldots, C_{l}: \\ C_{1}\ldots C_{l}C_{l+1}T'\text{ } HHL}}(-1)^{\ell(C_{1})}t^{\text{inv}(C_{1}\ldots C_{l+1}T')}(1-t)^{\text{des}(C_{1}\ldots C_{l+1}T')}=
\]
\[
(-t)^{\sum_{i=2}^{l-1}(\hat{N}(i)+\check{N}(i))}(1-t)^{\sum_{i=2}^{l-1}p(i)}\sum_{\substack{C_{l}: \\ C_{l}C_{l+1}T'\text{ } HHL}}(-1)^{\ell(C_{l})}t^{\text{inv}(C_{l}C_{l+1}T')}(1-t)^{\text{des}(C_{l}C_{l+1}T')}
\]
Now use Proposition 4.6 to get the result. 
\end{proof}
\begin{definition}
Call the HHL filling $\widehat{T}=\widehat{C}_{1}\ldots \widehat{C}_{(\lambda+\rho)_{1}}$ the 
\emph{root}. 
\end{definition}
We are now ready to prove the main result. This result is the formula for the sum of all HHL fillings which when sorted column-wise give the same SSYT. We note here that the statistics are defined based on pivot intervals in the root filling defined above.  
\begin{theorem} Fix a SSYT $\sigma$ of shape $\lambda+\rho$. We have, 
\[
\sum_{T: T\in \emph{sort}^{-1}(\sigma)}(-1)^{\ell(T[1])}t^{\emph{inv}(T)}(1-t)^{\emph{des}(T)}=\]
\[ 
(-t)^{\sum_{i=2}^{(\lambda+\rho)_{1}}(\check{N}(i)+\hat{N}(i))}(1-t)^{\sum_{i=2}^{(\lambda+\rho)_{1}}p(i)} 
\]
if the Non-overlapping Condition is satisfied for all columns in the root $\widehat{C}_{1}\ldots\widehat{C}_{(\lambda+\rho)_{1}}=\widehat{T}$. Otherwise the sum is $0$. 
\end{theorem}
\begin{proof}
Use Proposition 4.7 with $l=(\lambda+\rho)_{1}$.
\end{proof}
\subsection{Equivalence of statistics} 
We now show that these statistics are the same as the boxed and circling rules defined in Section 2.6. In this subsection, we let $T$ be a SSYT. The first step is to show that if $n(T)>0$ then there is a pair of columns in the root filling, $\widehat{T}$, that do not satisfy the Non-overlapping Condition. Recall that $n(T)$ is positive if and only if the SSYT, $T$, is not strict. We then show that $z(T)$ is equal to the sum of the number of pivots between adjacent columns of $\widehat{T}$ and that $l(T)$ is equal to the number of entries whose value lies in a pivot interval. 
\begin{definition} The root tableau $\widehat{T}$ is \emph{disjoint} if we have Non-overlapping Condition for all pairs of adjacent columns. 
\end{definition}
\begin{prop}
$n(T)>0$ if and only if $\widehat{T}$ is not disjoint. 
\end{prop}
\ytableausetup{nosmalltableaux}
\begin{proof}
First assume $n(T)>0$. Locate a boxed, circled configuration in $T$. Suppose the configuration is between columns $i$ and $i+1$. We have the following set up:
\[
\begin{ytableau}
a& \none&\none[|_{j-1}|_{j}] &\none&c \\
b&\none& \none[|_{j}\text{ }|_{j+1}]& \none&d \\
\end{ytableau}.
\]
So $[a,c]\cap [b,d]\ne \emptyset$. We can bring this to a two column HHL configuration by applying transpositions as follows:
\[
\begin{ytableau}
x & a \\
\none & \none \\
a & c \\
b &d
\end{ytableau} \longrightarrow
\begin{ytableau}
a & a \\
\none & \none \\
x & c \\
b &d
\end{ytableau}
\]
Since $x<a$ we have $[x,c]\cap [b,d] \ne \emptyset$. Apply such moves until we have a pivot entry to the left of $c$. Proceed to entry $b$ and apply similar such moves. So we have $x \leq a$ and $y \leq b$ for some $y$ and so $[x,c] \cap [y,d] \ne \emptyset$. If $c$ or $d$ are not distinct entries with respect to the left column then apply transpositions to the left column to match these entries. This places $x$ and $y$ in rows weakly below $c$ and $d$. So we get $[x,w] \cap [y,z] \ne \emptyset$ where $w$ and $z$ are not in the left column and $c \leq w$ and $d \leq z$. 
\newline \indent $[x,w] \cap [y,z] \ne \emptyset$ implies that if $x$ is matched with $w$ and $y$ is matched with $z$ in the root filling then we are done. If not, then there exists a legal transposition between pivot entries which implies the Non-overlapping Condition is not satisfied by Lemma 4.1. So $\widehat{T}$ is not disjoint. 
\newline \indent Now assume $n(T)=0$. Consider any adjacent pair of columns in $T$. Find the highest distinct entry in the left column. We have,
\[
\begin{ytableau}
a & c \\
b & d
\end{ytableau}
\]
where $a$ is distinct and $c \leq b$ since $n(T)=0$. If $c=b$ switch $a$ and $b$. If $c<b$ then $c$ is distinct since if not we have 
\[
\begin{ytableau}
a & c \\
b & d \\
\none \\
c
\end{ytableau}.
\]
However this is a semistandard two column configuration, so we arrive at a contradiction. 
\newline\indent We next show that if $c<b$ then $a$ is matched with $c$ in the root configuration. Since $a$ is the highest distinct entry in the left column then $a$ must have its match weakly above $c$ in the right column. Assume its strictly above. Call this entry $x$. We have,
\[
\begin{ytableau}
\none & x \\
\none\\
a &c \\
b& d
\end{ytableau}.
\]
Consider the entry to the left of $x$ and call it $y$. If $y$ is not distinct, find its match in the right column. This process terminates with finding a distinct entry above $a$ since in the left column since we have a finite number of rows. This gives a contradiction since $a$ is the highest distinct entry.
\newline \indent In both cases we end up matching $a$ with its pair, $x$ to get 
\[
\begin{ytableau}
a & x \\
\none \\
u & v
\end{ytableau}
\]
with $x<u$ since $x$ is distinct with respect to the right column and $n(T)=0$. Furthermore, $[a,x] \cap [u,v] = \emptyset$ for all rows $\ytableaushort{uv}$ below $x$. Go to the next highest distinct entry in the left column and repeat. This shows the Non-overlapping Condition is satisfied for all columns. 
\end{proof}
\begin{prop} Let $T$ be a strict SSYT. Let $\widehat{T}$ be the root. We have, 
\[
(1-t)^{z(T)}(-t)^{l(T)}=(1-t)^{\sum_{i=2}^{(\lambda+\rho)_{1}}p(i)}(-t)^{\sum_{i=2}^{(\lambda+\rho)_{1}}(\check{N}(i)+\hat{N}(i))+n-a_{0}}.
\]
where the left hand side is defined for $T$ and the right hand side is defined for $\widehat{T}$.
\end{prop}
\begin{proof}
Assume $n(T)=0$. We first show $z(T)=\sum_{i=2}^{(\lambda+\rho)_{1}}p(i)$. Suppose we have a separating wall that is not boxed, not circled, shown below:
\[
\begin{ytableau}
\none & \none &\none &\none[i-1]& \none &\none &\none[i]\\
\none & \none &\none &\none[\downarrow]&\none&\none&\none[\downarrow]\\
\none & \none &\none & a& \none[\ldots]&\none[|_{k}]&\text{\tiny{k+1}}&\ldots&\text{\tiny{k+1}}&\none[\text{ }|_{k+1}]\\
\text{\tiny{k+1}}&\none[\text{ } |'_{k+1}]&\none[\ldots] &b
\end{ytableau}.
\]
Since $|'_{k+1}$ is strictly to the left of $|_{k}$ then there is no $k+1$ entry in column $i-1$, since the entry directly below $a$ must be strictly larger than $k+1$, i.e. $b>k+1$. So $k+1$ is in column $i$ and $k+1$ is not in column $i-1$. This implies that there exists a descent in root configuration $\widehat{T}$.
\newline\indent Now suppose we have a descent in $\widehat{T}$, $\ytableaushort{ac}$. In addition assume we have no entries $b$ such that $a<b<c$. So we have
\[
\ytableaushort{ac}
\]
\[
\Bigg\downarrow \text{sort}
\]
\[
\ytableaushort{ac,d}
\]
with $a<d$. We then place separating walls to get,
\[
\ytableaushort{ac,d}\xrightarrow{\text{place separating walls}}\begin{ytableau}\none&a & \none[\ldots]&\none[|_{c-1}]&c&\none[|_{c}]\\ \none[|_{c}]& d\end{ytableau}.
\]
So we get a not boxed, not circled separating wall of index $c-1$. 
\newline\indent Now assume there exists some entries $b_{i}$ such that $a<b_{i}<c$. Since $c$ is not in column $i$ the entry below some $b_{i'}$, call it $x$, gives two cases, when $x>c$ and $x<c$. If $x>c$ then
\[
\begin{ytableau}
\none&\none[i-1]&\none&\none&\none&\none[i]\\
\none&\none[\downarrow]&\none&\none&\none&\none[\downarrow] \\
\none&a&\none&\none&\none&b_{i}\\
\none&\none[\vdots]& \none&\none&\none&\none[\vdots]\\
\none&b_{i'}&\none[|_{b_{i'}}]&\none[\ldots]&\none[|_{c-1}]&c&\none[\text{ } |_{c}]\\
\none[|_{c}\text{ }]& x
\end{ytableau}
\] 
which gives a not boxed, not circled configuration with respect to $|_{c-1}$. If $x<c$ then
\[
\begin{ytableau}
\none&\none[i-1]&\none&\none&\none&\none[i]\\
\none&\none[\downarrow]&\none&\none&\none&\none[\downarrow] \\
\none&a&\none&\none&\none&b_{i}\\
\none&\none[\vdots]& \none&\none&\none&\none[\vdots]\\
\none&b_{i'}&\none[|_{b_{i'}}]&\none[\ldots]&\none[|_{c-1}]&c&\none[\text{ } |_{c}]\\
\none& x&\none[|_{x}]&\none[\ldots]&\none[|_{c}]&d
\end{ytableau}.
\] 
Since $x<c$ then $b_{i'}<x<c$ which implies $b_{i'}<c-1$. It follows that we have $|_{c-2}|_{c-1}$ above $|_{c-1}|_{c}$ which implies that $n(T)>0$, which is a contradiction. 
\newline \indent So we have shown that there is a not boxed, not circled separating wall configuration if and only if there is a descent in the columns surrounding the separating walls. 
\newline\indent Next we show,
 \[
l(T)=\sum_{i=2}^{(\lambda+\rho)_{1}}(\check{N}(i)+\hat{N}(i))+n-a_{0}.
\]
To see this, we first show that the SSYT, $T$ has the following structure:
\[
\begin{ytableau}
a&b_{0}\\
b_{0} &b_{1}\\
\none[\vdots]&\none[\vdots]\\
b_{m} & c
\end{ytableau}
\]
where $a<b_{0}<\ldots<b_{m}<c$. Suppose there exists a row such that we have 
\[
\begin{ytableau}
a & b \\
\circ & \star \\
b & c
\end{ytableau}
\]
which implies $a<\circ<b<\star<c$. It follows that $a<b-1$. After placing separating walls we have,
\[
\begin{ytableau}
a &\none[\ldots] &\none[|_{b-2}]&\none[|_{b-1}]& b\\
\circ & \none[\ldots] &\none[|_{b-1}] & \none[|_{b}] & \star \\
b &\none&\none&\none& c
\end{ytableau}
\]
giving $n(T)>0$ which is a contradiction. So we have the desired structure. After placing separating walls we get, 
\[
\begin{ytableau}
a&\none[\text{\hspace{4mm}}|_{b_{0}-1}]&\none&\none&b_{0}\\
b_{0} &\none[|_{b_{0}}]&\none[\text{\hspace{4mm}}|_{b_{1}-1}]&\none&b_{1}\\
\none[\vdots]&\none&\none&\none&\none[\vdots]\\
\none&\none&\none[\text{\hspace{3mm}}|_{b_{m}-1}]&\none&b_{m}\\
b_{m} &\none&\none[|_{b_{m}}]&\none&c
\end{ytableau}
\]
Since $m=N_{a,c}[C_{i_{>r(a)}}\cap C_{i-1}]+N_{a,c}[C_{i_{<r(a)}}\cap C_{i-1}]$ and the fact that we can extend this argument to any number of such sections gives  
\[
l(\overrightarrow{C}_{i-1}\overrightarrow{C}_{i})=\check{N}(i)+\hat{N}(i).
\]
Therefore,
\[
\sum_{i=2}^{(\lambda+\rho)_{1}}l(\overrightarrow{C}_{i-1}\overrightarrow{C}_{i})=l(T)-n+a_{0}
\]
where the $-n+a_{0}$ comes from not counting the walls placed at the beginning of the rows. The result then follows. 
\end{proof}

\section{Appendix}
In this section we display some examples of the generation algorithm.

\subsection{Generation algorithm examples}
Here we present three examples of the generation algorithm. We assign $t$-coefficients to the leaves of the tree, as well as assign $t$-coefficients to the internal nodes of the tree corresponding to our compression scheme.
\newline\textbf{Example 1:} By factoring out the contribution of $T'$ we simply need to consider a two column configuration. Let \ytableausetup{smalltableaux,centertableaux}$\widehat{C}C'=\ytableaushort{15,33,44,22,68}$. We follow the construction of $\beta$ to get, 
\[
\beta=
\begin{matrix}
(1<6)& \big|& (1<4) &\big| & (1<3) &\big| & (1<2).
\end{matrix}
\]
The generation tree is on a separate page.
\newline
\textbf{Example 2: } Once again we consider only a two column configuration. Let $\widehat{C}C'=\ytableaushort{14,25,36}$. Following the construction of $\beta$ we get, 
\[
\beta=
\begin{matrix}
(2<3) & (1<3) & \big| & (1<2)
\end{matrix}
\]
In this example we can apply $(1<3)(2<3)$ to $C$ to get $\ytableaushort{34,15,26}$, however these two transpositions are in the same row. As a result we do not generate each HHL filling uniquely. We can generate $\ytableaushort{34,15,26}$ by applying $(1<2)(1<3)$ to $C$. 
The generation tree is on a separate page. 
\newline
\textbf{Example 3 } In this example we start with a SSYT $T=\ytableaushort{122334,2344,345,45}$ and generate the entire preimage of the sort map according to the construction given in Section 5.3.


\newpage
\textbf{Generation Tree for Example 1: }
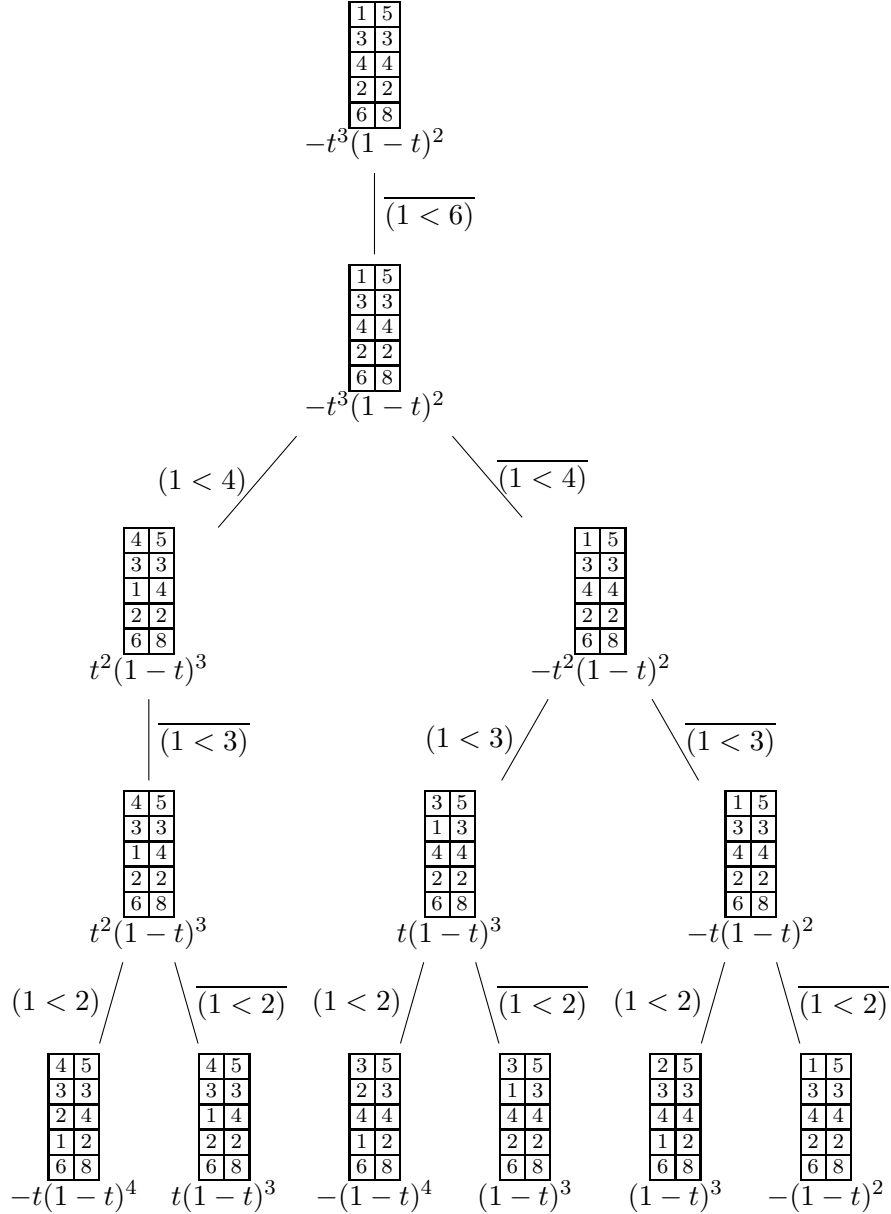
\[
\ytableausetup{smalltableaux}
\begin{tikzpicture}[node distance=0cm]
\node(3)[yshift=-7cm]{$\begin{matrix}\ytableaushort{15,33,44,22,68}\\ -t^{3}(1-t)^{2}\end{matrix}$};
\node(4)[yshift=-10.5cm]{$\begin{matrix}\ytableaushort{15,33,44,22,68}\\ -t^{3}(1-t)^{2}\end{matrix}$};
\node(5)[yshift=-14cm,xshift=3cm]{$\begin{matrix}\ytableaushort{15,33,44,22,68}\\ -t^{2}(1-t)^{2}\end{matrix}$};
\node(6)[yshift=-14cm,xshift=-3cm]{$\begin{matrix}\ytableaushort{45,33,14,22,68}\\ t^{2}(1-t)^{3}\end{matrix}$};
\node(7)[yshift=-17.5cm,xshift=5cm]{$\begin{matrix}\ytableaushort{15,33,44,22,68}\\ -t(1-t)^{2}\end{matrix}$};
\node(8)[yshift=-17.5cm,xshift=1cm]{$\begin{matrix}\ytableaushort{35,13,44,22,68}\\ t(1-t)^{3}\end{matrix}$};
\node(9)[yshift=-17.5cm,xshift=-3cm]{$\begin{matrix}\ytableaushort{45,33,14,22,68}\\ t^{2}(1-t)^{3}\end{matrix}$};
\node(10)[yshift=-21cm,xshift=6cm]{$\begin{matrix}\ytableaushort{15,33,44,22,68}\\ -(1-t)^{2}\end{matrix}$};
\node(11)[yshift=-21cm,xshift=4cm]{$\begin{matrix}\ytableaushort{25,33,44,12,68}\\ (1-t)^{3}\end{matrix}$};
\node(12)[yshift=-21cm,xshift=2cm]{$\begin{matrix}\ytableaushort{35,13,44,22,68}\\ (1-t)^{3}\end{matrix}$};
\node(13)[yshift=-21cm]{$\begin{matrix}\ytableaushort{35,23,44,12,68}\\ -(1-t)^{4}\end{matrix}$};
\node(14)[yshift=-21cm,xshift=-2cm]{$\begin{matrix}\ytableaushort{45,33,14,22,68}\\ t(1-t)^{3}\end{matrix}$};
\node(15)[yshift=-21cm,xshift=-4cm]{$\begin{matrix}\ytableaushort{45,33,24,12,68}\\ -t(1-t)^{4}\end{matrix}$};
\draw(3)--(4)node[midway, anchor=west]{$\overline{(1<6)}$};
\draw(4)--(5)node[midway, anchor=west]{$\overline{(1<4)}$};
\draw(4)--(6)node[midway, anchor=east]{$(1<4)$};
\draw(5)--(7)node[midway, anchor=west]{$\overline{(1<3)}$};
\draw(5)--(8)node[midway, anchor=east]{$(1<3)$};
\draw(6)--(9)node[midway, anchor=west]{$\overline{(1<3)}$};
\draw(7)--(10)node[midway, anchor=west]{$\overline{(1<2)}$};
\draw(7)--(11)node[midway, anchor=east]{$(1<2)$};
\draw(8)--(12)node[midway, anchor=west]{$\overline{(1<2)}$};
\draw(8)--(13)node[midway, anchor=east]{$(1<2)$};
\draw(9)--(14)node[midway, anchor=west]{$\overline{(1<2)}$};
\draw(9)--(15)node[midway, anchor=east]{$(1<2)$};

\end{tikzpicture}
\]
For the node below the edge labelled $\overline{(1<6)}$ we have 
\[
\begin{matrix}
(p<q)=(1<6)&\text{inv}(CC')=0 & \text{des}(CC')=2 \\ & m(CC',6)=N_{1,5}[C'_{>1}\cap C]+N_{6,6}[C'_{>5}\cap C]=3+0=3.
\end{matrix}
\]
For the node below the edge labelled $(1<4)$ we have 
\[
\begin{matrix}
(p<q)=(1<4)&\text{inv}(CC')=1 & \text{des}(CC')=3 \\ &m(CC',4)=N_{1,4}[C'_{>3}\cap C]+N_{6,4}[C'_{>5}\cap C]=1+0=1.
\end{matrix}
\]
For the node below the edge labelled $\overline{(1<3)}$ attached to the filling $CC'=\ytableaushort{15,33,44,22,68}$ we have 
\[
\begin{matrix}
(p<q)=(1<3)&\text{inv}(CC')=0 & \text{des}(CC')=2 \\ &m(CC',3)=N_{1,3}[C'_{>1}\cap C]+N_{6,3}[C'_{>6}\cap C]=1+0=1.
\end{matrix}
\]
\textbf{Generation Tree for Example 2: }
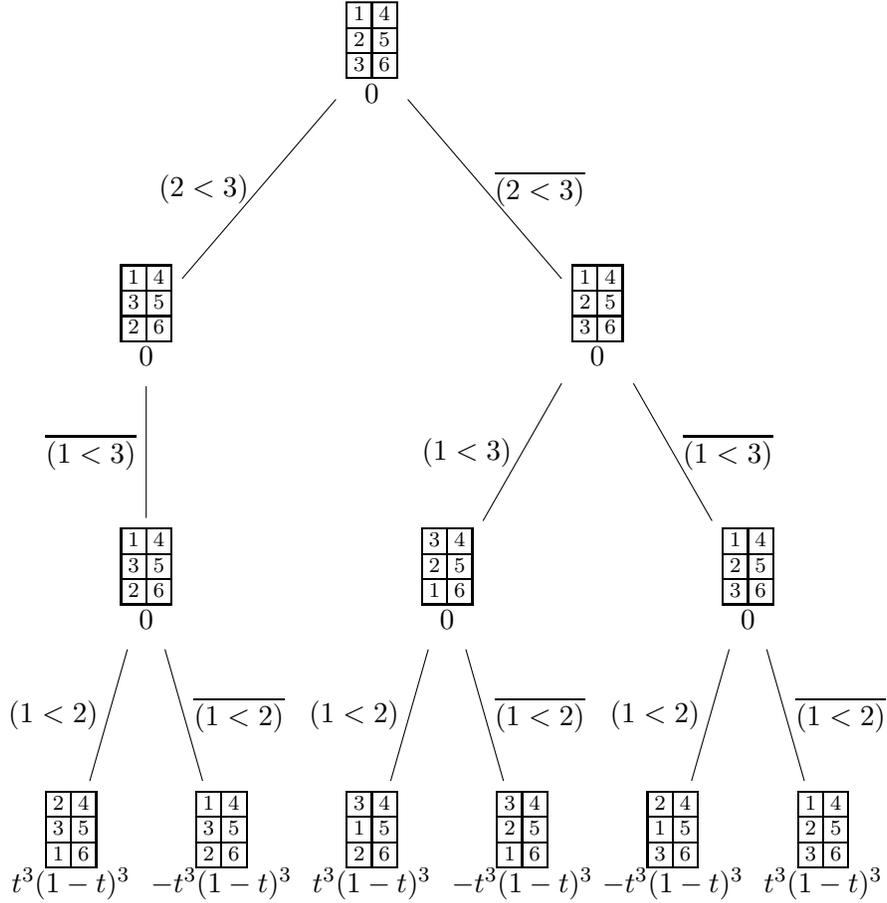
\[
\begin{tikzpicture}[node distance=0cm]

\node(1)[]{$\begin{matrix}\ytableaushort{14,25,36}\\ 0\end{matrix}$};
\node(2)[yshift=-3.5cm,xshift=3cm]{$\begin{matrix}\ytableaushort{14,25,36}\\ 0\end{matrix}$};
\node(3)[yshift=-3.5cm,xshift=-3cm]{$\begin{matrix}\ytableaushort{14,35,26}\\ 0\end{matrix}$};
\node(4)[yshift=-7cm,xshift=-3cm]{$\begin{matrix}\ytableaushort{14,35,26}\\ 0\end{matrix}$};
\node(5)[yshift=-7cm,xshift=5cm]{$\begin{matrix}\ytableaushort{14,25,36}\\ 0\end{matrix}$};
\node(6)[yshift=-7cm,xshift=1cm]{$\begin{matrix}\ytableaushort{34,25,16}\\ 0\end{matrix}$};
\node(7)[yshift=-10.5cm,xshift=-2cm]{$\begin{matrix}\ytableaushort{14,35,26}\\ -t^{3}(1-t)^{3}\end{matrix}$};
\node(8)[yshift=-10.5cm,xshift=-4cm]{$\begin{matrix}\ytableaushort{24,35,16}\\t^{3}(1-t)^{3} \end{matrix}$};
\node(9)[yshift=-10.5cm,xshift=2cm]{$\begin{matrix}\ytableaushort{34,25,16}\\ -t^{3}(1-t)^{3}\end{matrix}$};
\node(10)[yshift=-10.5cm]{$\begin{matrix}\ytableaushort{34,15,26}\\t^{3}(1-t)^{3} \end{matrix}$};
\node(11)[yshift=-10.5cm,xshift=4cm]{$\begin{matrix}\ytableaushort{24,15,36}\\ -t^{3}(1-t)^{3}\end{matrix}$};
\node(12)[yshift=-10.5cm,xshift=6cm]{$\begin{matrix}\ytableaushort{14,25,36}\\t^{3}(1-t)^{3} \end{matrix}$};
\draw (1)--(2)node[midway, anchor=west]{$\overline{(2<3)}$};
\draw (1)--(3)node[midway, anchor=east]{$(2<3)$};
\draw (3)--(4)node[midway, anchor=east]{$\overline{(1<3)}$};
\draw (2)--(5)node[midway, anchor=west]{$\overline{(1<3)}$};
\draw (2)--(6)node[midway, anchor=east]{$(1<3)$};
\draw (4)--(7)node[midway, anchor=west]{$\overline{(1<2)}$};
\draw (4)--(8)node[midway, anchor=east]{$(1<2)$};
\draw (6)--(9)node[midway, anchor=west]{$\overline{(1<2)}$};
\draw (6)--(10)node[midway, anchor=east]{$(1<2)$};
\draw (5)--(11)node[midway, anchor=east]{$(1<2)$};
\draw (5)--(12)node[midway, anchor=west]{$\overline{(1<2)}$};
\end{tikzpicture}
\]
\newpage
\textbf{Generation Tree for Example 3: }
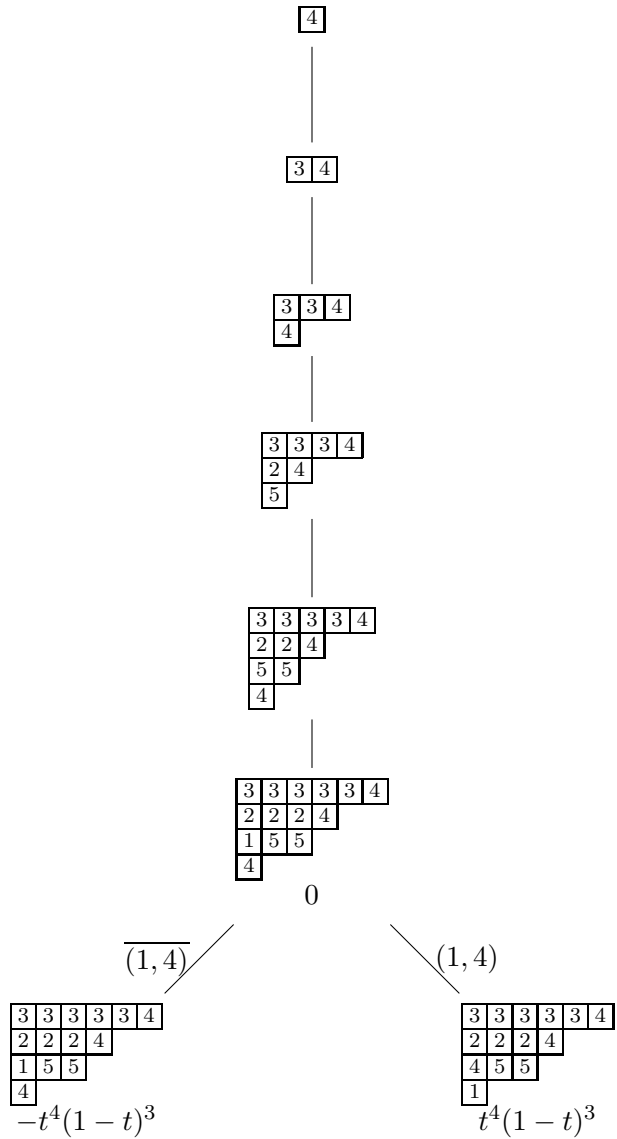
\[
\begin{tikzpicture}[node distance=0cm]
\node(1)[]{$\begin{matrix}\ytableaushort{4}\end{matrix}$};
\node(2)[yshift=-2cm]{$\begin{matrix}\ytableaushort{34}\end{matrix}$};
\node(3)[yshift=-4cm]{$\begin{matrix}\ytableaushort{334,4}\end{matrix}$};
\node(4)[yshift=-6cm]{$\begin{matrix}\ytableaushort{3334,24,5}\end{matrix}$};
\node(5)[yshift=-8.5cm]{$\begin{matrix}\ytableaushort{33334,224,55,4}\end{matrix}$};
\node(6)[yshift=-11cm]{$\begin{matrix}\ytableaushort{333334,2224,155,4}\\ 0\end{matrix}$};
\node(7)[xshift=3cm,yshift=-14cm]{$\begin{matrix}\ytableaushort{333334,2224,455,1}\\t^{4}(1-t)^{3}\end{matrix}$};
\node(8)[xshift=-3cm,yshift=-14cm]{$\begin{matrix}\ytableaushort{333334,2224,155,4}\\-t^{4}(1-t)^{3}\end{matrix}$};
\draw(1)--(2)node[]{};
\draw(2)--(3)node[]{};
\draw(3)--(4)node[]{};
\draw(4)--(5)node[]{};
\draw(5)--(6)node[]{};
\draw(6)--(7)node[midway, anchor=west]{$(1,4)$};
\draw(6)--(8)node[midway, anchor=east]{$\overline{(1,4)}$};
\end{tikzpicture}
\]

\bibliographystyle{alpha}

\begin{thebibliography}{LNS{\etalchar{+}}12b}
\bibitem[BeBF11]{BBF1}
J.~ Beineke, B. ~Brubaker, and S. ~Frechette.
\newblock Weyl group multiple Dirichlet series of type $C$,
\newblock {\emph{Pacific J. Math.}}, 254:11-46, 2011.
\bibitem[BrBF11]{BBF}
B.~ Brubaker, D.~ Bump, and S. ~Friedberg.
\newblock \emph{Weyl Group Multiple Dirichlet Series Type $A$ Combinatorial Theory},
\newblock Princeton University Press, Princeton, New Jersey, 2011.
\bibitem[BBL15]{BBL}
B. ~Brubaker, D. ~Bump, and A. ~Licata.
\newblock Iwahori Whittaker functions and Demazure operators,
\newblock \emph{J. Number Theory}, 146:41-68, 2015.
\bibitem[B]{Bump}
D.~ Bump.
\newblock Whittaker functions and representations. \newline
\newblock http://sporadic.stanford.edu/bump/whittaker/whittaker.html
\bibitem[CS80]{casups}
W. ~Casselman and J. ~Shalika.
\newblock The unramified principal series of {$p$}-adic groups. {II}. {T}he
{W}hittaker function.
\newblock {\em Compositio Math.}, 41:207--231, 1980.
\bibitem[FZ16]{FZ}
S.~ Friedberg and L. ~Zhang.
\newblock Tokuyama-Type formulas for type $B$,
\newblock \emph{Isr. J. Math.}, 216:617-655, 2016.
\bibitem[GRV15]{G}
V.~ Gupta, U. ~Roy, and R. ~Van Peski.
\newblock A generalization of Tokuyama's formula to the Hall-Littlewood polynomials,
\newblock \emph{Electron. J. Comb.}, 22(2), https://doi.org/10.37236/4732, 2015.
\bibitem[HHL05]{HHL0}
J. ~Haglund, M. ~Haiman, and N. ~Loehr.
\newblock A combinatorial formula for Macdonald polynomials,
\newblock \emph{J. Amer. Math. Soc.} 18:735-761,
2005.
\bibitem[HHL08]{HHL1}
J.~ Haglund, M. ~Haiman, and N.~ Loehr.
\newblock A combinatorial formula for non-symmetric Macdonald polynomials,
\newblock\emph{Am. J. Math.}, Vol. 130, No. 2:359-383, 2008.
\bibitem[HK02]{HK}
A.~ Hamel and R. ~King.
\newblock Symplectic Shifted Tableaux and Deformations of Weyl's Denominator Formula for sp(2n). 
\newblock\emph{J. of Algebraic Combin.} 16, 269–300, 2002.
\bibitem[IK13]{IK}
I. ~Klostermann.
\newblock Generalization of the Macdonald formula for Hall-Littlewood polynomials,
\newblock \emph{Adv. Math.} 248:1366-1403, 2013.
\bibitem[LLL19]{CL4}
K. ~Lee, C. ~Lenart, and D. ~Liu (with Appendix by D.~ Muthiah and A. ~Pusk\'{a}s).
\newblock Whittaker functions and Demazure characters,
\newblock{\em J. Inst. Math. Jussieu}, 18:759-781, 2019.
\bibitem[L05]{CL6}
C.~ Lenart and A.~Postnikov.
\newblock Affine Weyl groups in K-theory and representation theory,
\newblock \emph{Int. Math. Res. Not.} pages 1–65, 2007. Art. ID rnm038.
\bibitem[L07]{CL5}
C. ~Lenart.
\newblock On the combinatorics of crystal graphs, I. Lusztig's involution,
\newblock \emph{Adv. Math.} 211:204-243, 2007.
\bibitem[L09]{CL2}
C.~ Lenart.
\newblock On combinatorial formulas for Macdonald polynomials,
\newblock \emph{Adv. Math.} 220:324-340, 2009.
\bibitem[L10]{CL7}
C.~ Lenart.
\newblock Haglund-Haiman-Loehr type formulas for Hall-Littlewood polynomials of type $B$ and $C$,
\newblock \emph{Algebra Number Theory} 4:887-917, 2010.
\bibitem[L11]{CL1}
C. ~Lenart (Appendix with A.~ Lubovsky).
\newblock Hall-Littlewood polynomials, alcove walks, and fillings of Young diagrams,
\newblock  \emph{Discrete Math.} 311:258-275, 2011.
\bibitem[LL15]{lalgam}
C.~Lenart and A.~Lubovsky.
\newblock A generalization of the alcove model and its applications.
\newblock {\em J. Algebraic Combin.}, 41:751--783, 2015.
\bibitem[M95]{IG}
I.G. ~Macdonald.
\newblock \emph{Symmetric Functions and Hall Polynomials},
\newblock Oxford University Press, New York, Second Edition, 1995.
\bibitem[M03]{IG1}
I.G. ~Macdonald.
\newblock \emph{Affine Hecke Algebras and Orthogonal Polynomials},
\newblock Cambridge University Press, New York, 2003.
\bibitem[M11]{McN}
P. J. ~McNamara.
\newblock Metaplectic Whittaker functions and crystal bases,
\newblock \emph{Duke Math. J.}, 156 (2011), 1-31.
\bibitem[OS18]{OS}
D. ~Orr and M. ~Shimozono.
\newblock Specializations of non-symmetric Macdonald-Koornwinder polynomials,
\newblock \emph{J. of Algebraic Combin.} 47:91-127, 2018.
\bibitem[P16]{AP}
A.~Pusk{\'a}s. 
\newblock Gelfand-Tsetlin coefficients on Young tableaux,
\newblock Preprint, 2016.
\bibitem[P\textsuperscript{+}16]{AP1}
A. ~Pusk{\'a}s.
\newblock Whittaker functions on metaplectic covers of GL(r),
\newblock {\tt arXiv:1605.05400}, 2016.
\bibitem[RY11]{RY}
A. ~Ram and M. ~Yip. 
\newblock A combinatorial formula for Macdonald polynomials. 
\newblock {\em Adv. Math.}, 226:309-331, 2011.
\bibitem[S20]{SA}
 J. ~Saied.
\newblock A combinatorial formula for Sahi, Stokman, and Venkateswaran's generalization of Macdonald polynomials,
\newblock {\tt arXiv:2006.15086}, 2020.
\bibitem[S76]{shiefc}
T.~Shintani.
\newblock On an explicit formula for class-{$1$} ``{W}hittaker functions'' on
{$GL_{n}$} over {$P$}-adic fields.
\newblock {\em Proc. Japan Acad.}, 52:180--182, 1976.
\bibitem[T88]{T} 
T. ~Tokuyama. 
\newblock A generating function of strict Gel{'}fand patterns and some formulas on characters of general linear groups. 
\newblock \emph{J. Math. Soc. Japan}. 40:671-685, 1988.

\end{thebibliography}

\newcommand{\etalchar}[1]{$^{#1}$}

\end{document}